\documentclass{amsart}
\usepackage{amsfonts,amsmath,amstext,amssymb,amsthm, mathrsfs, amscd,enumerate,verbatim}
\usepackage[all]{xy}
\newtheorem{thm}{Theorem}[section]
\newtheorem{cor}[thm]{Corollary}
\newtheorem{ex}[thm]{Example}
\newtheorem{lem}[thm]{Lemma}
\newtheorem{prop}[thm]{Proposition}
\newtheorem{rem}[thm]{Remark}

\theoremstyle{definition}
\newtheorem{defn}[thm]{Definition}

\numberwithin{equation}{section}
\newcommand{\Hom}[3]{\operatorname{Hom}_{#1}(#2,#3)}

\newcommand{\ext}[4]{\operatorname{Ext}^{#1}_{#2}(#3,#4)}
\newcommand{\cdd}[3]{\operatorname{cd}_{#1}(#2,#3)}
\newcommand{\T}[2]{\operatorname{T}(#1,#2)}

\newcommand{\Ht}[2]{\operatorname{ht}_{#1}{(#2)}}

\newcommand{\h}[3]{\operatorname{H}^{#1}_{#2}(#3)}
\newcommand{\gam}[2]{\Gamma_{#1}(#2)}

\newcommand{\ann}{\operatorname{Ann}}
\newcommand{\ara}{\operatorname{ara}}
\newcommand{\V}{\operatorname{V}}

\newcommand{\zdv}{\operatorname{Zdv}}
\newcommand\supp {\operatorname{Supp}}
\newcommand\spec{\operatorname{Spec}}
\newcommand\ass {\operatorname{Ass}}
\newcommand\mass {\operatorname{MinAss}}
\newcommand\msupp {\operatorname{MinSupp}}
\newcommand\assh {\operatorname{Assh}}
\newcommand\att{\operatorname{Att}}
\newcommand\matt{\operatorname{MinAtt}}
\newcommand{\grad}[3]{\operatorname{depth}_{#1}(#2,#3)}
\newcommand{\fgrad}[3]{\operatorname{f-grade}(#1,#2,#3)}
\newcommand\depth{\operatorname{depth}}

\newcommand\fa{\mathfrak a}
\newcommand\fb{\mathfrak b}

\newcommand\fm{\mathfrak m}
\newcommand\fn{\mathfrak n}
\newcommand\fp{\mathfrak p}
\newcommand\fq{\mathfrak q}

\newcommand\N{\mathbb N}
\swapnumbers

\begin{document}

\title[Annihilator of Top Local Cohomology and Lynch's Conjecture]{ Annihilator of Top Local Cohomology  and Lynch's Conjecture}%
\author[A.~fathi]{Ali Fathi}
\address{Department of Mathematics, Zanjan Branch,
Islamic Azad University,  Zanjan, Iran.}
\email{alif1387@gmail.com}

%\date{*}%
\dedicatory{Dedicated to Professor Hossein Zakeri}%
%\commby{*}%
% ----------------------------------------------------------------
\begin{abstract}
Let $R$ be a commutative Noetherian ring, $\mathfrak a$ a proper  ideal of $R$ and $N$ a non-zero finitely generated  $R$-module with $N\neq \mathfrak a N$. Let $d$ (respectively $c$) be the smallest (respectively greatest)  non-negative integer $i$ such that the local cohomology $\operatorname{H}^i_{\mathfrak a}(N)$  is non-zero.  In this paper, we provide   sharp bounds under inclusion for the annihilators of  the  local cohomology modules
$\operatorname{H}^d_{\mathfrak a}(N)$, $\operatorname{H}^c_{\mathfrak a}(N)$ and  these annihilators are computed  in certain cases. Also,   we construct a counterexample to Lynch's conjecture.
 \end{abstract}
\maketitle
\section{\bf Introduction}
 Throughout this paper, $R$ is a commutative Noetherian ring with non-zero identity.  Let $\fa$ be an ideal of $R$, $N$ an $R$-module and $i$  a non-negative integer. The {\it $i$-th local cohomology } of  $N$ with respect to   $\fa$ was defined by Grothendieck as follows:
$$\h i{\fa}N:={\underset{n\in\N}{\varinjlim}\operatorname{Ext}^i_R(R/\fa^n,  N)};$$
see \cite{bs} and \cite{g} for   more details.

 Throughout this section, let $\fa$ be an ideal of $R$ and $N$  a finitely generated $R$-module.
  We recall that the cohomological dimension (respectively, the depth) of $N$ with respect to $\fa$, denoted by $\cdd R {\fa}N$ (respectively, $\grad{R}{\fa}N$),  is defined as the supremum (respectively, infimum) of the non-negative integers $i$ such that $\h i{\fa}N$ is non-zero.
   The arithmetic rank of $\fa$, denoted by $\textrm{ara}(\fa)$, is the least number of elements of $R$ required to generate an ideal which has the same radical as $\fa$. The $N$-height of $\fa$ is defined as  $\Ht N{\fa}:=\inf\{\dim_{R_\fp}(N_\fp): \fp\in\supp_R(N)\cap\V(\fa)\}$, where $\V(\fa)$ denotes the set of all prime ideals of $R$ containing $\fa$.  We denote the set of minimal elements of $\ass_R(N)$ by $\mass_R(N)$; also, the set of elements $\fp$ of $\ass_R(N)$ with $\dim_R(R/\fp)=\dim_R(N)$ is denoted by $\assh_R(N)$.
   For a  submodule $L$   of $N$ and $\fp\in\supp_R(N)$, we denote the contraction of $L_\fp$ under the canonical map $N\rightarrow N_\fp$ by $C^N_\fp(L)$. Finally, we denote the set of  integers (respectively, positive integers, non-negative integers) by $\mathbb Z$ (respectively, $\N$, $\N_0$). For any unexplained notation and terminology,
we refer the reader to \cite{bs}, \cite{bh} and \cite{mat}.

  We adopt the convention that the intersection (respectively,  union) of  empty family of  subsets of a set $A$ is $A$ (respectively,  the empty set). Also, we adopt the convention that the infimum (respectively, supremum) of  empty set of integers is $\infty$ (respectively, $-\infty$).

Since $\bigcup_{i\in\N_0}\supp_R(\h i{\fa}N)=\supp_R(N/\fa N)$(see \cite[Lemma 2.4]{f2018}), we obtain  $N=\fa N$ if and only if $\h i{\fa}N=0$ for all $i\in\N_0$. In this case we have $\cdd R{\fa}N=-\infty$, $\grad{R}{\fa}N=\infty$ and $\Ht N{\fa}=\infty$ by above convention.

If $R$ is a regular local ring containing a field and $\fa$ a proper ideal of $R$, then it is known that $\h i{\fa}R\neq 0$ if and only if $\h i{\fa}R$ is faithful (i. e., $\ann_R (\h i{\fa}R)=0$).  This was proved by Huneke and Koh in prime characteristic (see \cite[Lemma 2.2]{hk}) and by Lyubeznik in characteristic zero (see \cite[Corrolary 3.5]{ly} and the proof of \cite[Theorem 2.3]{hk}).  Boix and Eghbali provided a characteristic-free proof of this result   in \cite[Theorem 3.6]{be}. This  leads to the following conjecture of Lynch; see \cite[Conjecture 1.2]{l}.

\textbf{Lynch's Conjecture.} If $R$ is a  local ring and $\fa$ a proper ideal of $R$ with $c:=\cdd R {\fa}R>0$, then $\dim_R(R/\ann_R (\h c{\fa}R))=\dim_R(R/\gam {\fa}R)$. In particular, if $\fa$ contains a non-zerodivisor, then $\dim_R(R/\ann_R (\h c{\fa}R))=\dim_R(R)$.

    Let $\fa\neq R$ and $c:=\cdd R {\fa}R$.
     The conjecture is known to be false:  the first counterexample   was constructed  by Bahmanpour
    in \cite[Example 3.2]{b2017} over a  nonequidimensional  local ring of  dimension at least $5$ with cohomological dimension $c=2$. Also,  Singh and Walther
     \cite{sw} provided  a counterexample over a nonequidimensional local ring of dimension $3$ with cohomological dimension $c=2$. In both examples  $\ann_R (\h c{\fa}R)$
      has height zero.  Datta,  Switala and  Zhang, in \cite{dsw}, produced  a counterexample for this conjecture over a regular local ring of mixed characteristic such that
      $\ann_R (\h c{\fa}R)$ is non-zero and $\cdd R {\fa}R=\ara (\fa)$. There are some affirmative answers to Lynch's conjecture.  When   $R$ is either a ring of a prime
      characteristic and $\cdd R {\fa}R=\ara (\fa)$ or that $R$ is pure in a regular ring containing a field, Hochster and Jeffries  proved that  $\ann_R (\h c{\fa}R)$ has
       height zero; see
       \cite[Corollary 2.7 and Theorem 2.9]{hj}.  Also, we note that Datta-Switala-Zhang's example shows that the  result of Hochster and Jeffries does not hold for an arbitrary commutative Noetherian  ring $R$ and ideal $\fa$ with $\cdd R {\fa}R=\ara (\fa)$. Note that if $\dim_R(R/\ann_R (\h c{\fa}R))=\dim_R(R)$, then $\Ht R{\ann_R (\h c{\fa}R)}=0$ and if in addition $c>0$, then $\gam {\fa}R\subseteq\ann_R (\h c{\fa}R)$ and so $\dim_R(R/\ann_R (\h c{\fa}R))=\dim_R(R/\gam {\fa}R)=\dim_R(R)$.

We assume for the remainder of this section that $N\neq\fa N$, $c:=\cdd R{\fa}N$ and   $0=N_1\cap\dots\cap N_n$ is a minimal primary decomposition of the zero submodule of $N$
with $\ass_R(N/N_i):=\{\fp_i\}$ for all $1\leq i\leq n$.
    If $c=\dim_R(N)$, then  we have
  \begin{equation}{\textstyle\ann_R (\h c{\fa}N)=\ann_R (N/{\bigcap_{\cdd R {\fa}{R/\fp_i}=c}N_i})}.\end{equation}
  This equality was proved by Lynch \cite[Theorem 2.4]{l} whenever $R$ is a complete local ring and $N=R$. Bahmanpour et al. \cite[Theorem 1.1]{bag} proved it when $R$ is a complete local ring and $\fa$ is the maximal ideal of $R$. Finally,   it is proved in general form  by Atazadeh et al. in \cite[Theorem 2.3]{asn}. Therefore in the case $c=\dim_R(N)$,  Lynch's conjecture is true and the $N$-height  of $\ann_R (\h c{\fa}N)$ is zero.  Also recently, in \cite[Theorem 1.2]{an2022}, Atazadeh and Naghipour  show that  the $N$-height of $\ann_R (\h c{\fa}N)$ is zero in the case when $c=\dim_R(N)-1$.
  If $c$ is not necessarily equal $\dim_R(N)$,  in \cite [Theorem 3.4]{f}, we gave the following bound for the annihilator of $\h c{\fa}N$:
    \begin{equation}\textstyle{\ann_R ({N}/{\bigcap_{\fp_i\in\Delta}N_i})\subseteq\ann_R (\h c{\fa}N)
 \subseteq \ann_R ({N}/{\bigcap_{\fp_i\in\Sigma}N_i})},\end{equation}
 where  $\Delta:=\{\fp\in\ass_R(N): \cdd R {\fa}{R/\fp}=c\}$ and $\Sigma:=\{\fp\in\ass_R(N): \cdd R {\fa}{R/\fp}=\dim_R (R/\fp)=c\}$.
If $c=\dim_R(N)$, then $\Delta=\Sigma$ and this bound yields the  equation (1.1). This paper is divided into 5 section.

   In Sec. 2, for a submodule $L$ of $N$ and $\fp\in\supp_R(N)$,  we present some  properties of $C^N_\fp(L)$ which are used  in the sequal.

  In Sec. 3, for an arbitrary non-negative integer $t$, we provide a lower bound for the annihilator of $\h t{\fa}N$. More precisely, we show that
   $C^t(\fa, N):=\bigcap_{\fp_i\in\Delta(t)}N_i=\bigcap_{\fp\in\Delta(t)}C^N_\fp(0)=\gam {\fa(t)}N$ is the largest submodule $L$ of $N$ such that
   $\cdd R{\fa}L<t$, where $\Delta(t):=\{\fp\in\ass_R(N): \cdd R{\fa}{R/\fp}\geq t\}$ and $\fa(t):=\bigcap_{\fp\in\ass_R(N)\setminus\Delta(t)}\fp$.
   There is   a following lower bound for the annihilator of $\h t{\fa}N$
   \begin{equation}
   \ann_R(N/C^t(\fa, N))\subseteq \ann_R(\h t{\fa}N);
   \end{equation}
    see Theorem \ref{lower}.
   Now set $d:=\grad{R}{\fa}N$. We denote $C^{d}(\fa, N)$ and $C^{c}(\fa, N)$ by ${\rm S}(\fa, N)$ and $\T {\fa}N$ respectively.
   For each $t\leq d$, we have $C^t(\fa, N)={\rm S}(\fa, N)$ and for each $t\geq c+1$ we have $C^t(\fa, N)=N$ and there is the filtration
  \begin{equation}{\rm S}(\fa, N)=C^{d}(\fa, N)\subseteq\dots\subseteq C^{c}(\fa, N)=\T {\fa}N\subset N\end{equation}
  of submodules of $N$ such that, for each $d\leq t\leq c$,
  $\cdd R{\fa}{C^{t+1}(\fa, N)/C^t(\fa, N)}=t$ whenever $C^{t}(\fa, N)\neq C^{t+1}(\fa, N)$; see Proposition \ref{prop5}. The submodule $\T {\fa}N$ (respectively ${\rm S}(\fa, N)$) of $N$ is used in Sec. 4 (respectively Sec. 5)  to study the annihilator of the last (respectively first) non-zero local cohomology module of $N$ with respect to $\fa$. For a submodule $L$ of $N$, we have $L=\fa L$ if and only if $L\subseteq {\rm S}(\fa, N)$. We can regard this as a version of Nakayama's Lemma because ${\rm S}(\fa, N)=0$ if and only if $1-a$ is a non-zerodivisor on $N$ for all $a\in \fa$; see Proposition \ref{prop5}.

  In Sec. 4,  we consider the annihilator of the top local cohomology $\h c{\fa}N$. In Theorem \ref{annh2}, the upper bound  for the annihilator of $\h c{\fa}N$ in (1.2) is improved  as  follows.
\begin{equation} \textstyle{\ann_R(N/\T{\fa}N)\subseteq\ann_R (\h c{\fa}N)\subseteq \ann_R(N/\bigcap_{\fp\in\Sigma}C^N_\fp(0))},\end{equation}
where $\Sigma:=\{\fp\in\supp_R(N): \cdd R {\fa}{R/\fp}=\dim_R(R/\fp)=c\}$. Also, it is  proved that if for each $\fp\in\ass_R(N)$ with $\cdd R{\fa}{R/\fp}=c$ there exists $\fq\in\Sigma$ such that $\fp\subseteq\fq$, then  the above upper and lower bounds  are equal. By using this theorem, we construct a counterexample to Lynch's conjecture (see Example \ref{exa}) which extends   Bahmanpour's example  \cite{b2017} and   Singh--Walther's example \cite{sw}. We also conclude from this theorem  (see Corollary \ref{cor1}) that  if  $\Sigma\neq\emptyset$ or $c=\dim_R(N)-1$, then
 \begin{equation}
 \Ht N{\ann_R (\h c{\fa}N)}=0. \end{equation}

 Next,    it is shown in Theorem \ref{annh4} that if $(0:_{\h c{\fa}{N/N_i}}\fa)$ is finitely generated for all $\fp_i\in\ass_R(N)$ with $\cdd R{\fa}{R/\fp_i}=c$, then
\begin{equation}\ann_R (\h c{\fa}N)=\ann_R (N/\T{\fa}N).\end{equation}
In particular, if $N$ is coprimary and $(0:_{\h c{\fa}N}\fa)$ is finitely generated, then we have $\ann_R(\h c{\fa}N)=\ann_R (N)$. Therefore if $R$ is domain and $(0:_{\h c{\fa}R}\fa)$ is finitely generated, then $\h c{\fa}R$ is faithful.
As an  application of Theorem \ref{annh4}, we will prove in Corollary \ref{cor2}  that the equality (1.7) holds in  the following cases:
(i) $c\leq 1$;
(ii) $\dim_R(N/\fa N)\leq 1$;
(iii) $\dim_R(N)\leq 2$;
 (iv) $\h c{\fa}N$ is $\fa$-cofinite minimax.
%(v) $\h j{\fa}N$ is minimax for all $j\neq c$.
Also, if $(R, \fn)$ is a complete local ring, $\h c{\fa}N$ is  Artinian and $(0:_{\h c{\fa}N}\fa)$ is finitely generated, then it is proved in Lemma 4.9 that
\begin{align}
&\att_R(\h c{\fa}N)=\{\fp\in\ass_R(N): \cdd R{\fa}{R/\fp}=c\}\\
&\nonumber=\{\fp\in\mass_R(N): \surd(\fa+\fp)=\fn,\, \dim_R(R/\fp)=c\}.
\end{align}

Now, let $(R, \fn)$ be a local ring. Then $\dim_R(N)-\dim_R(N/\fa N)$ is a lower bound for the cohomological dimension $c:=\cdd R{\fa}N$ of $N$ with respect to $\fa$. If $\dim_R(N)-\dim_R(N/\fa N)=c$, then  we show in Theorem \ref{cd2} that
\begin{align}
&\dim_R(N)=\dim_R(R/\ann_R(\h c{\fa}N)),\\
&\nonumber\textstyle{\ann_R(\h c{\fa}N)\subseteq\ann_R(N/\bigcap_{\fp\in\assh_R(N)}C^N_\fp(0))}\end{align}
and  equality holds if $N$ is unmixed, that is, $\dim_R(R/\fp)=\dim_R(N)$ for all $\fp\in\ass_R(N)$. Note that the above theorem gives an affirmative answer to Lynch's conjecture. Then we deduce in Corollary \ref{sys} that if
$0\leq t\leq \dim_R(N)$ and $x_1,\dots,x_t\in\fn$ is a part of a system of parameters for $N$, then
 \begin{align}
 &\cdd R{(x_1,\dots,x_t)}N=t,\\
\nonumber &\dim_R(R/\ann_R(\h t{(x_1,\dots,x_t)}N))=\dim_R(N),\\
\nonumber&\textstyle{\ann_R(\h t{(x_1,\dots,x_t)}N)\subseteq\ann_R(N/\bigcap_{\fp\in\assh_R(N)}C^N_\fp(0))}\end{align}
and equality holds if $N$ is unmixed. Hence in this case also Lynch's conjecture holds.

In Sec. 5, we consider the annihilator of the first non-zero local cohomology $\h d{\fa}N$, where $d:=\grad{R}{\fa}N$.  More precisely, we show in Theorem \ref{annh7} that
\begin{equation}\textstyle{\ann_R(N/{\rm S}(\fa, N))\subseteq\ann_R(\h d{\fa}N)\subseteq\ann_R(N/\bigcap_{\fp\in\Sigma} C^N_\fp(0))\cap(\bigcap_{\fp\in\Sigma'}\fp)},\end{equation}
where   $\Sigma:=\{\fp\in\mass_R(N): \Ht {R/\fp}{(\fa+\fp)/\fp}=d\}$ and $\Sigma':=\{\fp\in\ass_R(N)\setminus\mass_R(N): \Ht {R/\fp}{(\fa+\fp)/\fp}=d\}$. In Lemma \ref{loc}, for an arbitrary non-negative integer $t$, when $(R,\fn)$ is local we improve the upper bound that presented in \cite[Theorem 3.2]{f}  for the annihilator $\h t{\fn}N$ as follows
\begin{equation}
\textstyle{\ann_R(\h t{\fn}N)\subseteq\ann_R(N/\bigcap_{\fp\in\Sigma(t)}C_\fp^N(0))\cap(\bigcap_{\fp\in\Sigma'(t)}\fp),}
\end{equation}
where $\Sigma(t):=\{\fp\in\mass_R(N): \dim_R(R/\fp)=t\}$ and $\Sigma'(t):=\{\fp\in\ass_R(N)\setminus\mass_R(N): \dim_R(R/\fp)=t\}$. The  last inclusion in (1.11) is equality whenever $N$ is  Cohen-Macaulay; see Corollary \ref{cor3}. Note that for $\fp\in\supp_R(N)$, $\ann_R(N/C^N_\fp(0))=C^R_\fp(\ann_R(N))\subseteq\fp$. Example \ref{exa2} shows that to improve the  upper bound for the annihilator of $\h d{\fa}N$ in (1.11), we can not replace $\mass_R(N)$ by $\ass_R(N)$ in the index set $\Sigma$. This example also shows that, in general,  there is not a subset $\Sigma$ of $\supp_R(N)$ such that $\ann_R(\h d{\fa}N)=\ann_R(N/\bigcap_{\fp\in\Sigma}C^N_\fp(0))$ even if $R$ is a complete regular local ring and $\fa$ is its maximal ideal. Also, in Lemma \ref{annh6}, when $N$ is coprimary, we show that
\begin{equation}
\ann_R(\h {\Ht N{\fa}}{\fa}N)=\ann_R(N).
\end{equation}
In particular, if $R$ is domain, then $\h {\Ht R{\fa}}{\fa}R$ is faithful.

In Proposition \ref{prop3}, it is proved that if $(R, \fn)$ is a homomorphic image of a Cohen-Macaulay local ring and $t\in\N_0$ is such that
$\h t{\fn}N\neq 0$, then
\begin{equation}\dim_R(R/\ann_R(\h t{\fn}N))\leq t\end{equation}
and equality holds if $\dim_R(R/\fp)=t$ for some $\fp\in\ass_R(N)$.
Example \ref{exa3} shows that  there is a local ring $(R, \fn)$ which is a homomorphic image of a complete regular local ring  such that
 \begin{equation}
 \dim_R(R/\ann_R(\h {\depth_R(R)}{\fn}{R}))<\depth_R(R)=\depth_R(R/\gam {\fn}R).
 \end{equation}
 Therefore  the analogue version of Lynch's conjecture is not true for $\h {\depth_R(R)}{\fn}R$ and inequality (1.14) may be strict.
\section{\bf Preliminaries}

Let $L$ be a proper submodule of an $R$-module $N$. Then $L$ is called a {\it primary submodule} of $N$ when for all $r\in R$ and $x\in N$ if $rx\in L$, then $x\in L$ or $r^nN\subseteq L$ for some $n\in\N$. If $L$ is a primary submodule of $N$, then $\fp:=\surd(\ann_R(N/L))$
is a prime ideal of $R$ and $L$ is called a $\fp$-primary submodule of $N$.  An expression of $L$ as an intersection of finitely many primary submodules of $N$ is called a {\it primary decomposition} of $L$ in $N$.
Such a  primary decomposition
 $$L=N_1\cap\dots \cap N_n\quad \textrm {with } N_i\   \textrm{ $\fp_i$-primary in } N\ (1\leq i\leq n)$$
of $L$ in $N$ is said to  be {\it minimal primary decomposition} when  $\fp_1,\dots,\fp_n$ are distinct and $\bigcap_{1\leq j\neq i\leq n}N_j\nsubseteq N_i$ for all $1\leq i\leq n$. In this case, we have $\ass_R(N/L)=\{\fp_1,\dots, \fp_n\}$ and hence $n$ and the set $\{\fp_1,\dots, \fp_n \}$ are uniquely determined by a minimal primary decomposition of $L$ in $N$. When $\ass_R(N)$ has just one element (or equivalently $0$ is a primary submodule of $N$), then $N$ is called
{\it coprimary}. See \cite{at, mat} for more details about the primary decomposition of modules.

   Let $S$ be a multiplicatively  closed subset  of $R$ and $L$ a submodule of an $R$-module $N$.  We denote  the contraction of $S^{-1}L$ under the canonical map $N\rightarrow S^{-1}N$ by $C^N_S(L)$ (in \cite{f} it is denoted by $S_N(L)$). If $\fp\in\supp_R(N)$ and $S=R\setminus \fp$, we write $C_\fp^N(L)$ instead of $C^N_S(L)$. For an ideal $\fa$ and a prime ideal $\fp$ of $R$ we show $C_\fp^R(\fa)$ by $C_\fp(\fa)$.
   Also, a subset $\Sigma$ of $\ass_R(N)$ is called an {\it isolated subset} of $\ass_R(N)$ when it satisfies the following condition: if $\fq\in \ass_R(N)$ and $\fq\subseteq\fp$ for some $\fp\in\Sigma$, then $\fq\in\Sigma$.

\begin{lem}\label{cont2} Let $L$ be a  submodule of an $R$-module $N$. Let  $\Sigma$  be a finite subset of $\spec(R)$ and $S:=R\setminus\bigcup_{\fp\in\Sigma}\fp$. Then $C^N_S(L)=\bigcap_{\fp\in\Sigma} C^N_{\fp}(L)$.
\end{lem}
\begin{proof}
If $\Sigma=\emptyset$, then $S=R$ and so $S^{-1}(L)=S^{-1}(N)=0$, $C^N_S(L)=N$. On the other hand, the intersection of empty family of submodules of $N$ is $N$. Hence the
assertion holds in this case. Now assume that $\Sigma$ is not empty.
If $x\in C^N_S(L)$, then in $S^{-1}(N)$, we have $x/1=l/s$ for some $s\in S$ and $l\in L$. Now for each $\fp\in\Sigma$, it is easy to see that  $x/1=l/s$ in $N_{\fp}$. Therefore
$x\in C^N_{\fp}(L)$ for all $\fp\in\Sigma$ and so $C^N_S(L)\subseteq\bigcap_{\fp\in\Sigma} C^N_{\fp}(L)$.

To prove the reverse inclusion, assume that $x$ is an arbitrary element in $N$ with  $x\notin C^N_S(L)$ and it is sufficient for us to show that
 $x\notin \bigcap_{\fp\in\Sigma} C^N_{\fp}(L)$. Since $x/1\notin S^{-1}(L)$ in $S^{-1}(N)$, we have $sx/s=x/1\notin S^{-1}(L)$ for all
  $s\in S$. Therefore $sx\notin L$ for all $s\in S$. Consequently, $(L:_Rx)\cap S=\emptyset$, where $(L:_Rx)$ denotes the ideal $\{r\in R: rx\in L\}$ of $R$.
  It follows that $(L:_Rx)\subseteq\bigcup_{\fp\in\Sigma}\fp$ and so $(L:_Rx)\subseteq\fp$ for some $\fp\in\Sigma$ by the Prime Avoidance Theorem.
 Now in $N_{\fp}$, if $x/1=l/s$ for some $l\in L$ and some $s\in R\setminus \fp$, then
   there exists $t\in R\setminus \fp$ such that $tsx=tl$ and so
   $ts\in(L:_Rx)\subseteq\fp$, which is impossible. Therefore $x/1\notin L_{\fp}$ or equivalently $x\notin C^N_{\fp}(L)$, as required.
\end{proof}

\begin{prop}\label{prop1}
 Let $L$ be a proper submodule of a non-zero finitely generated $R$-module $N$ and let $L=N_1\cap\dots\cap N_n$ be a minimal primary decomposition of $L$ in $N$
 with $\ass_R(N/N_i):=\{\fp_i\}$ for all $1\leq i\leq n$. Let $\Sigma$ be an isolated subset of $\ass_R(N/L)$. Then
 $$\textstyle{\bigcap_{\fp_i\in\Sigma}N_i=C^N_S(L)=\bigcap_{\fp\in\Sigma}C^N_\fp(L)=\bigcup_{t\in\N}(L:_N\fb^t)},$$ where $S:=R\setminus \bigcup_{\fp\in\Sigma}\fp$ and $\fb:=\bigcap_{\fp\in\ass_R(N/L)\setminus\Sigma} \fp$.
\end{prop}
\begin{proof}
If $\Sigma=\emptyset$, then $S=R$ and $\fb=\surd(\ann_R(N/L))$. Thus $\bigcap_{\fp_i\in\Sigma}N_i=C^N_S(L)=\bigcap_{\fp\in\Sigma}C^N_\fp(L)=\bigcup_{t\in\N}(L:_N\fb^t)=N$. So assume that $\Sigma\neq \emptyset$. The first claimed equality is proved in \cite[Lemma 2.2]{f}  and the second equality is proved in Lemma \ref{cont2}. Now we show that $\bigcap_{\fp_i\in\Sigma}N_i=\bigcup_{t\in\N}(L:_N\fb^t)$. Assume that $x\in \bigcup_{t\in\N}(L:_N\fb^t)$. Therefore  $\fb^tx\subseteq L$ for some $t\in\N$ and so $\fb^tx\subseteq N_i$ for all $1\leq i\leq n$. Let $\fp_i\in\Sigma$. Since $N_i$ is a $\fp_i$-primary submodule of $N$, it follows from  $\fb^tx\subseteq N_i$ that either $x\in N_i$ or $\fb^t\subseteq \fp_i$. If $\fb^t\subseteq \fp_i$, then $\fp\subseteq \fp_i$ for some $\fp\in\ass_R(N/L)\setminus \Sigma$. Since $\Sigma$ is an isolated subset of $\ass_R(N/L)$, we obtain $\fp\in\Sigma$, which is impossible. Therefore $x\in N_i$ and so
 $\bigcup_{t\in\N}(L:_N\fb^t)\subseteq\bigcap_{\fp_i\in\Sigma}N_i$. To prove the reverse inclusion, assume that $x\in\bigcap_{\fp_i\in\Sigma}N_i$ and we show that $x\in\bigcup_{t\in\N}(L:_N\fb^t)$. For each $1\leq i\leq n$, $\fp_i=\surd(\ann (N/N_i))$ and hence there exists $r_i\in\N$ such that $\fp_i^{r_i}N\subseteq N_i$. Set $r=\max \{r_i: \fp_i\in\ass_R(N/L)\setminus\Sigma\}$. Then $\fb^rN\subseteq\bigcap_{\fp_i\in\ass_R(N)\setminus\Sigma}N_i$ and so  $\fb^rx\subseteq\bigcap_{\fp_i\in\ass_R(N/L)\setminus\Sigma}N_i$. Also  $\fb^rx\subseteq\bigcap_{\fp_i\in\Sigma}N_i$ because $x\in\bigcap_{\fp_i\in\Sigma}N_i$. Therefore
 $$\textstyle{\fb^rx\subseteq (\bigcap_{\fp_i\in\ass_R(N/L)\setminus\Sigma}N_i)\cap(\bigcap_{\fp_i\in\Sigma}N_i)=\bigcap_{\fp_i\in\ass_R(N/L)}N_i=L}$$
and consequently $x\in\bigcup_{t\in\N}(L:_N\fb^t)$. This completes  the proof.
\end{proof}

 Let $N$ be a non-zero finitely generated $R$-module, $\fp\in\supp_R(N)$ and $n\in\N$.
  Then we have
$$\mass_R(N/\fp^nN)=\msupp_R(N/\fp^nN)={\rm Min}({\rm V}(\fp)\cap\supp_R(N))=\{\fp\}.$$
Thus $\{\fp\}$ is an isolated subset of $\ass_R(N/\fp^nN)$. Therefore, by Proposition \ref{prop1}, the $\fp$-primary component of each minimal primary decomposition of $\fp^nN$ in $N$ is  $C^N_\fp(\fp^nN)$ and hence it is uniquely determined by $n$, $\fp$ and $N$.
 The  $\fp$-primary component of $\fp^nN$ in $N$  is called the {\it n-th symbolic power} of $\fp$ with respect to $N$, denoted by $(\fp N)^{(n)}$.
 %We note that, in general,  $\fp^nN$ need not be a $\fp$-primary submodule of $N$;  see \cite[Chapter I, Example 3]{nor}.

\begin{lem} \label{sym1}
Let $N$ be a non-zero finitely generated $R$-module and $\fp\in\supp_R(N)$.  Then
 $$\textstyle{C^N_\fp(0)=\bigcap_{n\in\N}(\fp N)^{(n)}=\bigcap_{L\in\mathcal P}L,}$$
 where  $\mathcal P$ denotes the set of all $\fp$-primary submodules of $N$.
%In particular, for any $\fp\in\spec(R)$
%$$C_\fp(0)=\bigcap_{n=1}^{\infty}\fp^{(n)}=\bigcap_{{\fq \textrm{ is a } \fp\textrm{-primary}}\atop{  \textrm{ ideal of  R}}}\fq.$$
\end{lem}
\begin{proof}
We prove the claimed equalities in some steps.

1) $\ker(N\rightarrow N_\fp)\subseteq \bigcap_{L\in\mathcal P}L$. Assume that $x\in\ker(N\rightarrow N_\fp)$ and $L$ is an arbitrary $\fp$-primary submodule of $N$. Since $x/1$ is zero in $N_\fp$, there exists $s\in R\setminus \fp$ such that $sx=0\in L$. As $s\notin \fp$ and $L$ is a $\fp$-primary submodule of $N$, we have $x\in L$.

2) $\bigcap_{L\in\mathcal P}L\subseteq \bigcap_{n\in\N}(\fp N)^{(n)}$. This inclusion is obvious because  $(\fp N)^{(n)}$ is a $\fp$-primary submodule of $N$ for all $n$. We recall that $(\fp N)^{(n)}=C^N_\fp(\fp^nN)$ is the unique $\fp$-primary component of every minimal primary decomposition of $\fp^nN$ in $N$.

3) $\bigcap_{n\in\N}(\fp N)^{(n)}\subseteq \ker(N\rightarrow N_\fp)$. Assume that $x\in\bigcap_{n\in\N}(\fp N)^{(n)}$. Then $x/1\in\bigcap_{n\in\N}(\fp^nR_\fp)N_\fp$.  Now  Krull's Intersection Theorem \cite[Theorem 8.10]{mat} implies that  $\bigcap_{n\in\N}(\fp^nR_\fp)N_\fp=0$ and so $x/1=0$. Therefore $x\in\ker(N\rightarrow N_\fp)$.
\end{proof}
\begin{lem}\label{cont} Let $L$ be a  submodule of a non-zero finitely generated $R$-module $N$. Then the following statements hold.
\begin{enumerate}[\rm(i)]
%\item If $L, N_1, N_2$ are $R$-modules such that $L\subseteq N_1\subseteq N_2$ and $\fp\in\supp_R(N_1)$, then $C^{N_1}_\fp(L)\subseteq C^{N_2}_\fp(L)$.
%\item If $L_1, L_2$ are submodules of an $R$-module $N$ such that $L_1\subseteq L_2$ and $\fp\in\supp_R(N)$, then $C^{N}_\fp(L_1)\subseteq C^{N}_\fp(L_2)$.
\item If  $\fp_1,\fp_2\in\supp_R(N)$  are such that $\fp_1\subseteq\fp_2$, then $C^{N}_{\fp_2}(L)\subseteq C^{N}_{\fp_1}(L)$.
\item If  $\fp\in\supp_R(N)$, then $\ann_R (N/C^N_\fp(L))=C_\fp(\ann_R (N/L))$. In particular, $\ann_R (N/C^N_\fp(0))=C_\fp(\ann_R (N))$.
\item If  $\fp\in\supp_R(N)$, then $C_\fp(\ann_R (N))\subseteq\fq$ for all $\fq\in\supp_R(N)$ with $\fq\subseteq\fp$. In particular, $\Ht N{C_\fp(\ann_R (N))}=0$.
\end{enumerate}
\end{lem}
\begin{proof}
The statement  (i) is  obvious. To prove (ii), assume that  $\fp\in\supp_R(N)$ and $r\in R$. Then  we have
\begin{align*}
&\textstyle{r\in\ann_R (N/C^N_\fp(L))\Leftrightarrow rN\subseteq C^N_\fp(L)\Leftrightarrow {r\over 1} N_\fp\subseteq L_\fp}\\
&\textstyle{\Leftrightarrow {r\over 1}\in\ann_{R_\fp}(N_\fp/L_\fp)\Leftrightarrow r\in C_\fp(\ann_R (N/L)).}
\end{align*}
(Note that $(\ann_R(N/L))_\fp=\ann_{R_\fp}(N_\fp/L_\fp)$ by \cite[Proposition 3.14]{at}.) Now we prove (iii). Let $\fp, \fq\in\supp_R(N)$ with $\fq\subseteq\fp$. Then, by (i), $C_\fp(\ann_R (N))\subseteq C_\fq(\ann_R (N))$. Also if $r\in C_\fq(\ann_R (N))$, then $(r/1)N_\fq=0$. Therefore $r/1$ is not a unit in $R_\fq$  and so $r\in\fq$. Thus $C_\fq(\ann_R (N))\subseteq\fq$. These  inclusions prove (iii).
\end{proof}

\begin{lem}\label{lem5}
Let $L$ be a submodule of a non-zero $R$-module $N$ and $\fp, \fq\in\supp_R(N)$ with $\fp\subseteq\fq$. Then $(C^N_\fp(L))_\fq=C^{N_\fq}_{\fp R_\fq}(L_\fq)$.
\end{lem}
\begin{proof}
Let $\alpha\in(C^N_\fp(L))_\fq$. Hence, in $N_\fq$, $\alpha=x/s$ for some $s\in R\setminus \fq$ and some $x\in C^N_\fp(L)$. Therefore, in $N_\fp$, $x/1=l/t$ for some $l\in L$ and some $t\in R\setminus \fp$. It follows that $xtt'=lt'$ for some $t'\in R\setminus \fp$. Since $tt'/1\in R_\fq\setminus \fp R_\fq$,  in  $(N_\fq)_{\fp R_\fq}$ we have
${x\over s}/{1\over 1}={tt'\over 1}{x\over s}/{tt'\over 1}={lt'\over s}/{tt'\over 1}\in (L_\fq)_{\fp R_\fq}.$
This means that  $\alpha=x/s$ is an element of $C^{N_\fq}_{\fp R_\fq}(L_\fq)$ and so $(C^N_\fp(L))_\fq\subseteq C^{N_\fq}_{\fp R_\fq}(L_\fq)$.

Now we prove the reverse inclusion. Assume that $\alpha\in C^{N_\fq}_{\fp R_\fq}(L_\fq)$. Hence $\alpha=x/s$ for some $x\in N$ and some $s\in R\setminus\fq$, and in $(N_\fq)_{\fp R_\fq}$ we have ${x\over s}/{1\over 1}\in (L_\fq)_{\fp R_\fq}$. Therefore there exists $l/s'\in L_\fq$ with $l\in L, s'\in R\setminus \fq$ and $t/s''\in R_\fq\setminus\fp R_\fq$ with $t\in R\setminus \fp, s''\in R\setminus \fq$ such that ${x\over s}/{1\over 1}={l\over s'}/{t\over s''}$ in $(N_\fq)_{\fp R_\fq}$. It follows that there exists $t'/s'''\in R_\fq\setminus \fp R_\fq$ with $t'\in R\setminus \fp, s'''\in R\setminus \fq$ such that
${x \over s}{t\over s''}{t'\over s'''}={l\over s'}{t'\over s'''}$ in $N_\fq$.
Thus $xtt's's'''s^{iv}=lt'ss''s'''s^{iv}$ for some $s^{iv}\in R\setminus \fq$. Since $R\setminus\fq\subseteq R\setminus \fp$,  $xt''=l'$ for some $t''\in R\setminus \fp$ and some $l'\in L$. Hence, in $N_\fp$ we have
$x/1=xt''/t''=l'/t''\in L_\fp$
and so $x\in C^N_\fp(L)$. Therefore  $\alpha=x/s\in (C^N_\fp(L))_\fq$ and hence $C^{N_\fq}_{\fp R_\fq}(L_\fq)\subseteq(C^N_\fp(L))_\fq$. This completes the proof.
\end{proof}
\section{\bf A lower bound for the annihilator of local cohomology}
Let $N$ be a finitely generated $R$-module,  $\fa$ an ideal of $R$  and $t$  an arbitrary non-negative integer.
In this section, we provide a lower bound for the annihilator of the local cohomology $\h t{\fa}N$; see Theorem \ref{lower}.

We recall that, the {\it cohomological dimension} of an $R$-module $N$  with respect to an ideal $\fa$  is defined as $\cdd R{\fa}N:=\sup\{i\in\N_0: \h i{\fa}N\neq 0\}$. The {\it arithmetic rank} of $\fa$, denoted by $\textrm{ara}(\fa)$, is the least number of elements of $R$ required to generate an ideal which has the same radical as $\fa$. By \cite[Corollary 3.3.3]{bs}, $\cdd R{\fa}N\leq \textrm{ara}(\fa)<\infty$ and hence $\cdd R{\fa}N\in \N_0\cup\{-\infty\}$. Also, it follows from \cite[Exercise 6.2.6 and Theorem 6.2.7]{bs} that $\h i{\fa}N=0$  for all $i\in\N_0$ if and only if $N=\fa N$. Therefore $\cdd R {\fa}N\in\N_0$ when $N\neq\fa N$ and $\cdd R{\fa}N=-\infty$ when $N=\fa N$.

\begin{lem}[See {\cite[Theorem 1.2]{dnt} or \cite[Proposition 4.7]{cjr}}]\label{cd}
  Let $N$ and $L$ be  two finitely generated $R$-modules and $\fa$ an ideal of $R$. If $\supp_R(N)\subseteq \supp_R(L)$, then $\cdd R {\fa}N\leq\cdd R {\fa}L$.
  In particular,
$\cdd R {\fa}N=\cdd R {\fa}L$ whenever $\supp_R(N)=\supp_R(L)$.
\end{lem}
  Let $N$ be a finitely generated $R$-module and $\fa$ an ideal of $R$. Since $\supp_R(N)=\supp_R(\bigoplus_{\fp\in\ass_R(N)}R/\fp)$, Lemma \ref{cd} implies that  $\cdd R {\fa}N=\max\{\cdd R {\fa}{R/\fp}: \fp\in\ass_R(N)\}$. We will often use   this  fact and Lemma \ref{cd}   without explicit  mention. We refer the reader to
   \cite[Sec. 4]{cjr} and \cite{dnt} for more details about the cohomological dimension.

\begin{thm}\label{lower} Let $N$ be a non-zero finitely generated $R$-module and $\fa$ an ideal of $R$. Let
  $0=N_1\cap\hdots\cap N_n$ be a minimal primary decomposition of the zero submodule of $N$ with $\ass_R(N/N_i):=\{\fp_i\}$ for all $1\leq i\leq n$. For each $t\in\mathbb Z$,  set $\Delta(t):=\{\fp\in\ass_R(N): \cdd R{\fa}{R/\fp}\geq t\}$.
Then the following statements hold.
\begin{enumerate}[\rm(i)]
\item There is the following equalities:
 $$\textstyle{\bigcap_{\fp_i\in\Delta(t)}N_i=\bigcap_{\fp\in\Delta(t)}C_\fp^N(0)=C^N_{S(t)}(0)=\gam {\fa (t)}N,}$$  where  $S(t):=R\setminus\bigcup_{\fp_i\in\Delta(t)}\fp$ and $\fa (t):=\bigcap_{\fp\in\ass_R(N),\ \cdd R{\fa}{R/\fp}<t}\fp$. In particular, $\bigcap_{\fp_i\in\Delta(t)}N_i$ is independent of the choice of minimal primary decomposition of the zero submodule of $N$.
  \item  For each submodule $L$ of $N$,  $\cdd R{\fa}L<t$ if and only if $L\subseteq C^N_{S(t)}(0)$.
  \item There is the following lower bound for the annihilator of $\h t{\fa}N$:
   $$\textstyle{\ann_R({N}/\bigcap_{\fp\in\Delta(t)}C_\fp^N(0))=\bigcap_{\fp\in\Delta(t)}C_\fp(\ann_R(N))\subseteq\ann_R(\h t{\fa}N).}$$
 \end{enumerate}
\end{thm}
\begin{proof}
 (i) If $\fq\in\ass_R (N)$ and $\fq\subseteq \fp$ for some $\fp\in\Delta(t)$, then  $t\leq\cdd R{\fa}{R/\fp}\leq\cdd R{\fa}{R/\fq}$ and so $\fq\in\Delta(t)$. Therefore
$\Delta(t)$ is an isolated subset of $\ass_R (N)$ and hence  (i) follows from  Proposition  \ref{prop1}.

(ii) Set $C=C^N_{S(t)}(0)$. Then, by (i), we have
\begin{align*}
&\ass_R (C)=\textstyle{\ass_R(\bigcap_{\fp_i\in\Delta(t)}N_i)=\ass_R(N)\setminus\Delta(t)}\\
&=\{\fp\in\ass_R (N): \cdd R{\fa}{R/\fp}<t\}.
\end{align*}
 Hence,   $\cdd R{\fa}{C}<t$ and so, for each submodule $L$ of $C$, $\cdd R{\fa}L\leq\cdd R{\fa}C<t$. Conversely,
   if $L$ is a submodule of $N$ such that $\cdd R{\fa}L<t$, then
$$\ass_R ({L}/({L\cap C}))=\ass_R (({L+C})/{C})\subseteq\ass_R ({N}/{C})=\Delta(t).$$
Thus if $\ass_R ({L}/({L\cap C}))\neq\emptyset$, then $t\leq\cdd R{\fa}{{L}/({L\cap C})}\leq\cdd R{\fa}L$, which is impossible. Therefore
$L\subseteq C$ and the proof of (ii) is completed.

(iii) By (ii),  $\cdd R{\fa}{C}<t$. Therefore $\h t{\fa}N\cong\h t{\fa}{N/C}$ and hence
 $$\ann_R(N/C)\subseteq\ann_R(\h t{\fa}{N/C})=\ann_R(\h t{\fa}N).$$
  \end{proof}
% Note that we don't need to consider the case $N=\fa N$ separately in above theorem. However if $N=\fa N$, then $R=\surd(\fa+\ann_R(N))=\surd(\fa+\bigcap_{\fp\in\ass_R(N)}\fp)=\bigcap_{\fp\in\ass_R(N)}\surd(\fa+\fp)$
%and so $\fa+\fp=R$ for all $\fp\in\ass_R(N)$. Thus $\cdd R{\fa}{R/\fp}=-\infty$ for all $\fp\in\ass_R(N)$. Hence $\Delta(t)=\emptyset$ and $S(t)=\fa(t)=R$. It follows that
%$\bigcap_{\fp_i\in\Delta(t)}N_i=\bigcap_{\fp\in\Delta(t)}C_\fp^N(0)=C^N_{S(t)}(0)=\gam {\fa (t)}N=N$. This proves (i) in this case.  Also since $\cdd R{\fa}N=-\infty<t$, the statement (ii) holds. Finally, the left side of inclusion in (iii) is $\ann_R(N/N)=R$ which is a lower bound for the  annihilator of $\h t{\fa}N$  because $\h t{\fa}N=0$ for all $t$ by our assumption.
\begin{defn} \label{def1} Let $N$ be a non-zero finitely generated $R$-module, $\fa$ an ideal of $R$ and $t\in\N_0$. Let $0=N_1\cap\dots\cap N_n$ be a minimal primary decomposition of the zero submodule of $N$ with $\ass_R(N/N_i):=\{\fp_i\}$ for all $1\leq i\leq n$. We Set $\Delta(t):=\{\fp\in\ass_R(N): \cdd R{\fa}{R/\fp}\geq t\}$, $S(t):=R\setminus\bigcup_{\fp\in\Delta(t)}\fp$ and $\fa(t):=\bigcap_{\fp\in\ass_R(N)\setminus\Delta(t)}\fp$. Then, by Theorem \ref{lower},
 $$\textstyle{\bigcap_{\fp_i\in\Delta(t)}N_i=\bigcap_{\fp\in\Delta(t)}C_\fp^N(0)=C^N_{S(t)}(0)=\gam {\fa (t)}N}$$
 is the largest submodule $L$ of $N$ with the property that $\cdd R{\fa}L<t$.
 We denote this submodule of $N$ by $C^t(\fa, N)$. If $N=\fa N$, then $\Delta(t)=\emptyset$, $S(t)=R$  and $\fa(t)=\surd(\ann_R(N))$. Thus $C^t(\fa, N)=N$ for all $t$. If $N\neq\fa N$, then $d:=\grad{R} {\fa}N$ and $c:=\cdd R{\fa}N$   are finite non-negative integers and we denote $C^d(\fa, N)$ and $C^c(\fa, N)$ by ${\rm S}(\fa, N)$ and ${\rm T}(\fa, N)$ respectively.
 \end{defn}
In Proposition \ref{prop5}, we give some properties of $C^t(\fa, N)$. The following lemma is needed.
 \begin{lem}\label{grad}
 Let $N$ be a finitely generated $R$-module and $\fa$ an ideal of $R$ with $N\neq \fa N$. Then for each $\fp\in \ass_R(N)$ with $\fa+\fp\neq R$, there is the following inequalities:
 $$\grad{R} {\fa}N\leq\Ht {R/\fp}{\fa+\fp}\leq \cdd R{\fa}{R/\fp}.$$
 In particular, for each $\fp\in \ass_R(N)$, $\cdd R{\fa}{R/\fp}<\grad{R} {\fa}N$ if and only if $\cdd R{\fa}{R/\fp}=-\infty$ or equivalently $\fa+\fp=R$.
 \end{lem}
 \begin{proof}
 Assume that $\fp\in\ass_R(N)$ with $\fa+\fp\neq R$ and $\fq$ is a prime ideal of $R$ containing $\fa+\fp$ such that $\Ht {R/\fp}{\fa+\fp}=\Ht {R/\fp}{\fq/\fp}$. Then we have
 $$\grad{R} {\fa}N\leq\grad{R} {\fq}N\leq\depth_{R_\fq}({N_\fq}).$$
  Also,  depth formula \cite[Lemma 9.3.2]{bs} yields
  $$\depth_{R_\fq}({N_\fq})\leq\depth_{R_\fp}(N_\fp)+\Ht {R/\fp}{\fq/\fp}=\Ht {R/\fp}{\fq/\fp}=\Ht {R/\fp}{\fa+\fp}.$$
  These inequalities prove the first claimed inequality. To prove the  other inequality, we set $\bar{R}:=R/\fp$.
  Since $\fa\bar R$ is a proper ideal of $\bar R$, it follows from \cite[Exercise 7.3.4]{bs} and the Independence Theorem that
  $$\Ht{\bar{R}}{\fa \bar{R}}\leq\cdd {\bar R}{\fa \bar R}{\bar R}=\cdd R{\fa}{R/\fp}.$$
 This proves the second claimed inequality. The last assertion follows immediately from the first part.
 \end{proof}

 \begin{defn} Let $N$ be a finitely generated $R$-module and $\fa$  an ideal of $R$.  We say that $N$ is  relative Cohen-Macaulay with respect to $\fa$ if $\grad{R} {\fa}N=\cdd R{\fa}N$ or equivalently there is precisely one non-vanishing local cohomology
module of $N$ with respect to $\fa$; see \cite{ra}.
\end{defn}
 Note that if $N=\fa N$, then $\h i{\fa}N=0$ for all $i\in\N_0$ and hence $\grad{R} {\fa}N=\inf\{i\in\N_0: \h i{\fa}N\neq 0\}=\infty$ and $\cdd R{\fa}N=\sup\{i\in\N_0: \h i{\fa}N\neq 0\}=-\infty$. Therefore $\grad{R} {\fa}N\neq\cdd R{\fa}N$ and so $N$ is not relative Cohen-Macaulay with respect to $\fa$ in this case.
\begin{cor}\label{rel} Let $\fa$ be an ideal of $R$ and $N$  a finitely generated $R$-module. If $N$ is relative Cohen-Macaulay with respect to $\fa$, then, for each $\fp\in\ass_R(N)$ with $\fa+\fp\neq R$, there is the following equalities:
$$\grad{R} {\fa}N=\Ht {R/\fp}{\fa+\fp}=\cdd R{\fa}{R/\fp}=\cdd R{\fa}N.$$
\end{cor}
 \begin{proof}
 It is an immediate consequence of Lemmas \ref{grad} and  \ref{cd}.
 \end{proof}

\begin{prop}\label{prop5}
Let $N$ be a finitely generated $R$-module and $\fa$ an ideal of $R$ such that $N\neq\fa N$. Set $d:=\grad{R}{\fa}N$, $c:=\cdd R{\fa}N$ and, for each $t\in\mathbb Z$,  $\Delta(t):=\{\fp\in\ass_R(N): \cdd R{\fa}{R/\fp}\geq t\}$.
The following statements hold.
\begin{enumerate}[\rm(i)]
\item There is the following equalities:
\begin{align*}
&\ass_R(N/C^t(\fa, N))=\{\fp\in\ass_R(N): \cdd R{\fa}{R/\fp}\geq t\},\\
 &\ass_R(C^t(\fa, N))=\{\fp\in\ass_R(N): \cdd R{\fa}{R/\fp}<t\},\\
 &\ass_R(C^{t+1}(\fa, N)/C^t(\fa, N))=\{\fp\in\ass_R(N): \cdd R{\fa}{R/\fp}=t\}.
\end{align*}
In particular, $C^t({\fa}, N)=N$ if and only if $\Delta(t)=\emptyset$; and $C^t({\fa}, N)=0$ if and only if $\Delta(t)=\ass_R(N)$.
\item $C^t(\fa, N)=N$ for all $t>c$;  $C^t(\fa, N)={\rm S}(\fa, N)$ for all $t\leq d$; and
     $\{C^i(\fa, N)\}_{i\in\mathbb Z}$ gives  the following bounded ascending chain
     $${\rm S}(\fa, N)=C^d(\fa, N)\subseteq\dots\subseteq C^c(\fa, N)=\T{\fa}N\subsetneqq C^{c+1}(\fa, N)=N$$
      of submodules of $N$ such that, for each $d\leq t\leq c$, $\cdd R{\fa}{C^{t+1}(\fa, N)/C^t(\fa, N)}=t$ whenever $C^{t}(\fa, N)\neq C^{t+1}(\fa, N)$.
      \item  $\ass_R({\rm S}(\fa, N))=\{\fp\in\ass_R(N): \fa+\fp=R\}$ and $\ass_R(N/{\rm S}(\fa, N))=\{\fp\in\ass_R(N): \fa+\fp\neq R\}.$
      \item  ${\rm S}(\fa, N)=0$ if and only if for each $a\in\fa$, $1-a$ is a non-zerodivisor on $N$.
      \item For each submodule $L$ of $N$, $L=\fa L$ (or equivalently $\h i{\fa}L=0$ for all $i\in \N_0$) if and only if $L\subseteq{\rm S}(\fa, N)$.
      In particular, the following statements are equivalent.
      \begin{enumerate}[(1)]
      \item For each $a\in\fa$, $1-a$ is a non-zerodivisor on $N$.
        \item  For each submodule $L$ of $N$, $L=\fa L$ if and only if $L=0$.
        \end{enumerate}
\end{enumerate}
\end{prop}
\begin{proof}
(i) Let $0=N_1\cap\dots\cap N_n$ be a minimal primary decomposition of the zero submodule of $N$ with $\ass_R(N/N_i)=\{\fp_i\}$ for all $1\leq i\leq n$.  By definition, $C^t(\fa, N)=\bigcap_{\fp_i\in\Delta(t)}N_i$. Now by \cite[Lemma 2.1]{f}, we have
\begin{align*}
&\ass_R(N/C^t(\fa, N))=\Delta(t)=\{\fp\in\ass_R(N): \cdd R{\fa}{R/\fp}\geq t\},\\
 &\ass_R(C^t(\fa, N))=\ass_R(N)\setminus\Delta(t)=\{\fp\in\ass_R(N): \cdd R{\fa}{R/\fp}<t\}.
 \end{align*}
To prove the last claimed equality, assume  that $\fp\in\ass_R(C^{t+1}(\fa, N)/C^t(\fa, N))$. Since $C^{t+1}(\fa, N)/C^t(\fa, N)$ is a submodule of $N/C^t(\fa, N)$,  $\cdd R{\fa}{R/\fp}\geq t$. On the other hand, $\fp\in\supp_R(C^{t+1}(\fa, N))$ and so $\cdd R{\fa}{R/\fp}\leq\cdd R{\fa}{C^{t+1}(\fa, N)}< t+1$. Therefore $$\ass_R(C^{t+1}(\fa, N)/C^t(\fa, N))\subseteq\{\fp\in\ass_R(N): \cdd R{\fa}{R/\fp}=t\}.$$
Now we prove the reverse inclusion. Assume that $\fp\in\ass_R(N)$ and $\cdd R{\fa}{R/\fp}=t$. Hence $\fp\in\ass_R(N/C^t(\fa, N))\setminus\ass_R(N/C^{t+1}(\fa, N))$. It follows from the exact sequence $0\rightarrow C^{t+1}(\fa, N)/C^t(\fa, N)\rightarrow N/C^t(\fa, N)\rightarrow N/C^{t+1}(\fa, N)\rightarrow 0$ that $\fp\in\ass_R(C^{t+1}(\fa, N)/C^t(\fa, N))$. This proves the reverse  inclusion.

 (ii) If $t>c$, then $\Delta(t)=\emptyset$ and so $C^t(\fa, N)=N$ by (i). Also for $t\leq d$ and $\fp\in\ass_R(N)$, Lemma \ref{grad} implies that $\cdd R{\fa}{R/\fp}\geq t$ when $\fa+\fp\neq R$, and  $\cdd R{\fa}{R/\fp}=-\infty$ when $\fa+\fp=R$. Thus $\Delta(t)=\{\fp\in\ass_R(N): \fa+\fp\neq R\}=\Delta(d)$ and consequently $C^t(\fa , N)={\rm S}(\fa, N)$ for all $t\leq d$. Also, it is clear by definition that $\T {\fa}N\neq N$ and $C^t(\fa, N)\subseteq C^{t+1}(\fa, N)$ for all $t$. The final assertion follows from (i).

 (iii) By  Lemma \ref{grad},  we have $\Delta(d)=\{\fp\in\ass_R(N): \fa+\fp\neq R\}$. Now (iii) follows from (i).

(iv) Since $\zdv_R(N)=\bigcup_{\fp\in\ass_R(N)}\fp$,  $1-a$ is a non-zerodivisor on $N$ for all $a\in\fa$, if and only if $\fa+\fp\neq R$ for all $\fp\in\ass_R(N)$ or equivalently $\ass_R({\rm S}(\fa, N))=\emptyset$. This proves (iv).

(v) By (ii),  $C^0(\fa, N)={\rm S}(\fa, N)$. Therefore it follows from  Theorem \ref{lower}(ii) that for each submodule $L$ of $N$, $\h i{\fa}L=0$ for all $i\in\N_0$ (or equivalently $L=\fa L$) if and only if $L\subseteq {\rm S}(\fa, N)$. The last assertion follows from (iv). This completes the proof.
\end{proof}

\section{\bf An upper  bound for the annihilator of top local cohomology and Lynch's conjecture}
    In \cite[Theorem 3.4]{f}, we provide a bound for the annihilator of top local cohomology. In the following theorem, we establish a sharper upper bound for this annihilator. Using this theorem, we can compute the annihilators of top local cohomology modules  in certain cases and we are able to construct a counterexample to Lynch's conjecture.

\begin{thm} \label{annh2} Let $N$ be a finitely generated $R$-module and $\fa$ an ideal of $R$ with $N\neq\fa N$. Set $c:=\cdd R {\fa}N$, $\Delta:=\{\fp\in\ass_R(N): \cdd R {\fa}{R/\fp}=c\}$ and $\Sigma:=\{\fp\in\supp_R(N): \cdd R {\fa}{R/\fp}=\dim_R (R/\fp)=c\}$.
Then
  \begin{align*}  {\ann_R({N}/\operatorname{T}({\fa}, N))}\subseteq\ann_R(\h c{\fa}N)
\textstyle{\subseteq \ann_R (N/\bigcap_{\fp\in\Sigma}C^N_{\fp}(0))}.
 \end{align*}
 Moreover, if for each $\fp\in\Delta$ there exists $\fq\in\Sigma$ with $\fp\subseteq\fq$, then
 $$\textstyle{\ann_R(\h c{\fa}N)=\ann_R (N/\T {\fa}N)=\ann_R(N/\bigcap_{\fp\in\Sigma}C^N_{\fp}(0)).}$$
\end{thm}
\begin{proof}
The first claimed inclusion follows from Theorem \ref{lower}(iii). Now we prove the second inclusion. If $\Sigma=\emptyset$, then $\bigcap_{\fp\in\Sigma}C^N_{\fp}(0)=N$ and so there is nothing to prove. Hence assume that $\Sigma\neq\emptyset$ and $\fp\in\Sigma$. Let $L$ be an arbitrary $\fp$-primary submodule of $N$.  Since $\ass_R(N/L)=\{\fp\}$, we obtain
\begin{align*}
&\{\fq\in\ass_R(N/L): \cdd R {\fa}{R/\fq}=\cdd R {\fa}{N/L}\}=\{\fp\}\\&=\{\fq\in\ass_R(N/L): \cdd R {\fa}{R/\fq}=\dim_R(R/\fq)=\cdd R {\fa}{N/L}\}.
\end{align*}
Therefore $\ann_R (\h c{\fa}{N/L})=\ann_R (N/L)$ by \cite[Theorem 3.4(iii)]{f}.
 Also, the exact sequence $0\rightarrow L\rightarrow N\rightarrow N/L\rightarrow 0$ induces the epimorphism
$\h c{\fa}N\rightarrow \h c{\fa}{N/L}$ and so
$$\ann_R (\h c{\fa}{N})\subseteq\ann_R (\h c{\fa}{N/L})=\ann_R (N/L).$$
Since $\fp$ is an arbitrary element of $\Sigma$ and $L$ is an arbitrary $\fp$-primary submodule of $N$, it follows from the above inclusion in view of Lemma \ref{sym1}  that
\begin{align*}
&\ann_R (\h c{\fa}{N})\subseteq\textstyle{\bigcap_{\fp\in\Sigma}\, \bigcap_{L\in\mathcal P}\ann_R ({N/L})}
=\textstyle{\bigcap_{\fp\in\Sigma}\ann_R ({N/\bigcap_{L\in\mathcal P}L})}\\
&=\textstyle{\bigcap_{\fp\in\Sigma}\ann_R ({N/C^N_{\fp}(0)})}=\textstyle{\ann_R (N/\bigcap_{\fp\in\Sigma}C^N_{\fp}(0)),}
\end{align*}
where $\mathcal P$ denotes the set of all $\fp$-primary submodules of $N$. This proves the second claimed inclusion.

 Finally, assume that  $0=N_1\cap\dots\cap N_n$ is a minimal primary decomposition of the zero submodule of $N$ with $\ass_R(N/N_i):=\{\fp_i\}$ for all $1\leq i\leq n$ and assume that for each $\fp_i\in\Delta$ there exists $\fq_i\in\Sigma$ such that $\fp_i\subseteq\fq_i$. Since $N_i$ is a $\fp_i$-primary submodule of $N$, by Lemma \ref{sym1}, we have $C^N_{\fp_i}(0)\subseteq N_i$ and so $\bigcap_{\fq\in\Sigma}C^N_{\fq}(0)\subseteq C^N_{\fq_i}(0)\subseteq C^N_{\fp_i}(0)\subseteq N_i$.  Since $\fp_i$ is an arbitrary element of $\Delta$, we obtain $\bigcap_{\fq\in\Sigma}C^N_{\fq}(0)\subseteq \bigcap_{\fp_i\in\Delta}N_i=\T {\fa}N,$ see Definition \ref{def1}. Hence
 $\ann_R(N/\bigcap_{\fq\in\Sigma}C^N_{\fq}(0))\subseteq\ann_R(N/\T {\fa}N)$.
 Now the first part of theorem gives the claimed equalities and the proof is completed.
\end{proof}
Let  $N$ be a finitely generated $R$-module  of dimension $n\geq 1$, $\fa$ an ideal of $R$ with  $N\neq\fa N$ and $c:=\cdd R {\fa}N$.  In \cite[Theorems 1.1 and 1.2]{an2022}, Atazadeh and Naghipour as one of their main results proved   that $\Ht N{\ann (\h c{\fa}N)}<n-c$. In particular,  $\Ht N{\ann_R (\h c{\fa}N)}=0$ when $c=n-1$. In the  following corollary,  we prove their result by using Theorem \ref{annh2} and  we also show that if there exist $\fp\in\supp_R(N)$ with $\cdd R {\fa}{R/\fp}=\dim_R(R/\fp)=c$ ($c$ is not necessarily equal $n-1$), then $\ann_R (\h c{\fa}N$  has $N$-height  zero.

\begin{lem}\label{lem3} Let $N$ be a finitely generated $R$-module, $\fa$ an ideal of $R$ with $N\neq\fa N$ and $c:=\cdd R{\fa}N$. Then
  $\h c{\fa}{N\otimes_R(\cdot)}\cong \h c{\fa}N\otimes_R(\cdot)$ on the category of $R$-modules and $R$-homomorphisms.
\end{lem}
\begin{proof}
Set  $\bar{R}:=R/\ann_R(N)$. It follows from the Independence Theorem \cite[Theore 4.2.1]{bs} and Lemma \ref{cd} that $\cdd {\bar R}{\fa\bar{R}}{\bar{R}}=\cdd R {\fa}{\bar{R}}=\cdd R {\fa}{N}=c$. Hence $\h c{\fa\bar{R}}{\cdot}$ is a right exact functor on the category of $\bar{R}$-modules. Furthermore  $\h c{\fa\bar{R}}{\cdot}$ is an additive functor that preserves direct sums. Therefore $\h c{\fa\bar{R}}{\cdot}$ is naturally isomorphic to  $\h c{\fa \bar{R}}{\bar{R}}\otimes_{\bar{R}}(\cdot)$ on the category of $\bar{R}$-modules; see \cite[Theorem 5.45]{r}. Since $N$ is an $\bar R$-module, $N\otimes_RM$ has $\bar R$-module structure for each $R$-module $M$ and so, on the category of $R$-modules, we have
\begin{align*}
&\h c{\fa}{N\otimes_R(\cdot)}\cong\h c{\fa \bar{R}}{N\otimes_R(\cdot)}\cong N\otimes_R(\cdot)\otimes_{\bar{R}}\h c{\fa\bar{R}}{\bar{R}}\\
&\cong(\cdot)\otimes_RN\otimes_{\bar{R}}\h c{\fa\bar{R}}{\bar{R}}
\cong (\cdot)\otimes_R\h c {\fa\bar{R}}{N}
\cong(\cdot)\otimes_R\h c{\fa}N.
\end{align*}
\end{proof}
\begin{cor} \label{cor1} Let $\fa$ be an ideal of $R$ and  $N$  a finitely generated $R$-module of dimension $n$ such that $N\neq \fa N$. Let  $c:=\cdd R{\fa}{N}$  and $\fp\in\supp_R(N)$ is such that $\cdd R {\fa}{R/\fp}=c$.
\begin{enumerate}[\rm(i)]
\item If $\dim_R(R/\fp)=c$, then $\Ht N {\ann_R(\h {c}{\fa}N)}=0$.
\item If $\dim_R(R/\fp)>c$, then $\Ht N {\fp}\leq n-c-1$.
\end{enumerate}
In particular, if $n\geq 1$ and $c=n-1$, then $\Ht N {\ann_R (\h {c}{\fa}N)}=0$.
\end{cor}
\begin{proof}
 We set $J:=\ann_R (\h {c}{\fa}N)$. If $\dim_R(R/\fp)=c$, then $J\subseteq \ann_R (N/C^N_{\fp}(0))=C_\fp(\ann_R (N))$ by Theorem \ref{annh2}. Hence (i) is an immediate consequence of  Lemma \ref{cont}(iii). Now to prove (ii), suppose that $\dim_R(R/\fp)>c$. Therefore
$$c=\cdd R {\fa}{R/\fp}<\dim_R(R/\fp)\leq n-\Ht N {\fp}$$
and hence $\Ht N {\fp}<n-c$. This proves (ii). Finally, assume that $c=n-1$. By Lemma \ref{lem3}, $\h c{\fa}{N/J N}\cong\h c{\fa}N/J \h c{\fa}N$. Since $J\h c{\fa}N=0$,
 $\h c{\fa}{N/J N}\cong\h c{\fa}N$.
Therefore $\cdd R {\fa}{N/JN}=c$. Hence there exists $\fp\in\supp_R(N/JN)=\V(J)$ such that $\cdd R {\fa}{R/\fp}=c$. If $\dim_R(R/\fp)=c$, then $\Ht NJ=0$ by (i). Otherwise, by (ii),  $\Ht N{\fp}=0$. As  $J\subseteq\fp$, we obtain  $\Ht NJ=0$.
\end{proof}
\begin{rem}\label{rem2} Let $\fa, \fb$ be  ideals of $R$ and $N$ a finitely generated $R$-module. It follows from the Independence Theorem   that $\h i{\fa}N\cong\h i{\fb}N$ for all $i$ when
$\fa+\ann_R(N)=\fb+\ann_R(N)$; see \cite[Theorem 4.2.1]{bs}. This fact is used in Example \ref{exa}. Furthermore, we proved, in \cite[Theorem 2.2]{f2018}, that
\begin{align*}
&\inf\{i\in\N_0: \h i{\fa}N\ncong\h i{\fb}N\}=\fgrad {\fa+\fb}{\fa\cap\fb}N\\
&=\inf\{\depth {N_\fp}: \fp\in \V(\fa+\ann_R(N))\triangle\V(\fb+\ann_R(N))\},
\end{align*}
where $A\triangle B=A\cup B-A\cap B$ denotes the symmetric difference of the sets $A$ and $B$ and for ideals $\fa$ and $\fb$, $\fgrad{\fa}{\fb}N$ denotes
the $\fa$-filter grade of  $\fb$ on $N$; see \cite{f2015} for definition and basic properties. In particular, $\h i{\fa}N\cong\h i{\fb}N$ for all $i$ if and only if
$\surd(\fa+\ann_R(N))=\surd(\fb+\ann_R(N))$.
\end{rem}
Now, by using Theorem \ref{annh2}, we construct a counterexample to Lynch's conjecture  which extends  the examples given in  \cite{b2017} and \cite{sw}.

\begin{ex}\label{exa} Let $S$ be a commutative Noetherian ring of finite dimension $d\geq 3$ such that there exists a subset $U:=\{u_1,\dots, u_d\}$ of $S $ with $\Ht SU=d$ and suppose that  each ideal generated by a subset of $U$ is a prime ideal of $S$. (When $S$ is a local ring, then it is easy to see that $S$ must be  a regular ring with maximal ideal $(U)$  and $U$ is a regular system of parameters for $S$.)
 Let $X$, $Y$ and $Z$ be disjoint non-empty subsets of $U$ such that $|X|\leq|Y|\leq|Z|$ (here, for a set $A$, $|A|$ denotes the cardinal number of  $A$) and let $X'$ and  $Y'$ be non-empty subsets of $X$ and $Y$ respectively. Set $J:=(X)\cap (Y)\cap (Z)$, $R:=S/J$ and  $I:=(X')+(Y')+J/J$. Then
\begin{enumerate}[\rm(i)]
\item $\ass_R(R)=\{\fp_1:=(X)/J, \ \fp_2:=(Y)/J, \ \fp_3:=(Z)/J\}$.
\item $\grad{R} I{R/\fp_1}=\cdd R I{R/\fp_1}=|Y'|$, $\grad{R} I{R/\fp_2}=\cdd R I{R/\fp_2}=|X'|$ and $\grad{R} I{R/\fp_3}=\cdd R I{R/\fp_3}=|X'|+|Y'|$.
\item $\dim_R(R/\fp_1)=d-|X|$, $\dim_R(R/\fp_2)=d-|Y|$ and $\dim_R(R/\fp_3)=d-|Z|$.
\item $\gam IR=0$,  $\dim_R(R)=\dim_R(R/\gam IR)=d-|X|$ and $c:=\cdd R IR=|X'|+|Y'|$.
\item Set $q:=(U-X'\cup Y')/J$, then $\fp_3\subseteq\fq$ and $\cdd R I{R/\fq}=\dim_R(R/\fq)=c$.
\item  $\ann_R (\h cIR)=(Z)/J$ and $\dim_R(R/\ann_R (\h cIR)=d-|Z|$. In particular,  $$\dim_R(R/\gam IR)-\dim_R(R/\ann_R (\h cIR)=|Z|-|X|.$$
\end{enumerate}
\end{ex}
\begin{proof}
Since $\Ht SU=d$, Krull's Generalized Principal Ideal Theorem \cite[Theorem 13.5]{mat}  implies that $u_i\notin (U\setminus\{u_i\})$ for all $1\leq i\leq d$. For each subset $V$  of $U$, since $(V\setminus\{u_i\})$ is a prime ideal of $S$,  $\ass_S(S/(V\setminus\{u_i\}))=\{(V\setminus\{u_i\})\}$  and so $u_i$ is a non-zerodivisor on $S/(V\setminus\{u_i\})$ for all $1\leq i\leq d$ . Hence every permutation of $u_1,\dots,u_d$ is an $S$-sequence.

    Next, for each $1\leq i\leq d$,  since $u_i\notin(u_1,\dots,u_{i-1})$ and $(u_1,\dots,u_{i-1})$ is a prime ideal of $S$, we  have $\dim_S(S/(u_1,\dots,u_i))\leq \dim_S(S/(u_1,\dots,u_{i-1}))-1$. We can now repeat  this argument to deduce that  $\dim_S(S/(u_1,\dots,u_i))\leq \dim_S(S)-i=d-i$ for all $1\leq i\leq d$. On the other hand, the strict chain of prime ideals $(u_1,\dots,u_i)\subset\dots\subset(u_1,\dots,u_d)$ yields $\dim_S(S/(u_1,\dots,u_i))\geq d-i$. Therefore $\dim_S(S/(u_1,\dots,u_i))= d-i$ for all $1\leq i\leq d$. By renaming the elements of $U$, we deduce that $\dim_S(S/(V))=d-|V|$  for all subsets $V$ of $U$.  These  facts are used in the sequel.

(i) It is clear that $\fp_1\cap\fp_2\cap\fp_3=0$ is a minimal primary decomposition of $0$ in $R$ and $\ass_R(R)=\{\fp_1, \fp_2, \fp_3\}$.

(ii) It follows from the Independence Theorem (see  Remark \ref{rem2}) that
$$\h iI{R/\fp_1}\cong\h iI{S/(X)}\cong \h i{(X', Y')}{S/(X)}\cong\h i{(Y')}{S/(X)}.$$
We also have $$|Y'|=\grad S{(Y')}{S/(X)}\leq \cdd S {(Y')}{S/(X)}\leq \ara(Y')\leq |Y'|.$$
Therefore  $\h i{(Y')}{S/(X)}$ and consequently $\h iI{R/\fp_1}$ are non-zero only at  $i=|Y'|$.  Similarly, we have
$\h iI{R/\fp_2}\cong\h i{(X')}{S/(Y)}$  and $\h iI{R/\fp_3}\cong\h i{(X', Y')}{S/(Z)}.$
Thus $\h iI{R/\fp_2}$ is non-zero only at  $i=|Y'|$ and $\h iI{R/\fp_3}$ is non-zero only at  $i=|X'|+|Y'|$.

(iii) As was mentioned at the beginning of the proof, we have $\dim_R(R/\fp_1)=\dim_S(S/(X))=d-|X|$, $\dim_R(R/\fp_2)=\dim_S(S/(Y))=d-|Y|$ and $\dim_R(R/\fp_3)=\dim_S(S/(Z))=d-|Z|$.

(iv) Since $\spec(R)=\supp_R(\bigoplus_{\fp\in\ass_R(R)}R/\fp)$, we obtain
$$\textstyle{c:=\cdd R  IR=\cdd R  I{\bigoplus_{1\leq i\leq 3}R/\fp_i}=\max_{1\leq i\leq 3}\cdd R  I{R/\fp_i}=|X'|+|Y'|},$$
$$\textstyle{\dim_R(R)=\max_{1\leq i\leq 3}\dim_R({R/\fp_i})=d-|X|.}$$
Next,   as $X'$  and $Y'$ are nonempty  sets, there exists $x\in X'$ and $y\in Y'$. We have $x+y+J\in I\setminus \bigcup_{1\leq i\leq 3}\fp_i$ because $X, Y, Z$ are disjoint sets and $u_i\notin (U\setminus \{u_i\})$ for all $1\leq i\leq d$. Therefore  $x+y+J\in I$ is a non-zerodivisor on $R$ and hence $\gam IR=0$. This completes the proof of (iv).

(v) It is clear that $Z\subseteq U-X\cup Y\subseteq U-X'\cup Y'$ and consequently $\fp_3\subseteq \fq$. Also, it follows from the  Independence Theorem that (note that $J\subseteq (U-X'\cup Y')$)
$$\h iI{R/\fq}\cong\h iI{S/(U-X'\cup Y')}\cong\h i{(X', Y')}{S/(U-X'\cup Y')}.$$
Similar to (ii) we can deduce that $\h i{(X', Y')}{S/(U-X'\cup Y')}$ and consequently $\h iI{R/\fq}$ are non-zero only at $i=|X'|+|Y'|$. Therefore $\cdd RI{R/\fq}=c$. Finally, we have
$$\dim_R(R/\fq)=\dim_S(S/(U-X'\cup Y'))=d-|U-X'\cup Y'|=|X'|+|Y'|=c.$$

(vi) It follows from (v) and Theorem \ref{annh2}  that
$\ann_R (\h cIR)=\ann_R (R/\fp_3)=\fp_3$
and hence $\dim_R(R/\ann_R  (\h cIR))=d-|Z|$.  Therefore
$$\dim_R(R/\gam IR)-\dim_R(R/\ann_R  (\h cIR))=d-|X|-(d-|Z|)=|Z|-|X|.$$
\end{proof}
\begin{rem}
(i) (Bahmanpour's example \cite[Example 3.2]{b2017}). Let $S$ be a regular local ring of dimension $d\geq 7$ and $U:=\{u_1,\dots,u_d\}$ a system of parameters for $S$. Let  $l$ be an integer with  $7\leq l\leq d$. Set $X:=\{u_1, u_2\}$, $Y:=\{u_3, u_4\}$, $Z:=\{u_5,\dots, u_l\}$, $X':=\{u_1\}$ and $Y':=\{u_3\}$. Let
     $J:=(X)\cap(Y)\cap(Z)$, $R:=S/J$ and $I:=(X')+(Y')+J/J$. Then,  by  Example \ref{exa}, we have
          \begin{align*}
    & c:=\cdd R IR=|X'|+|Y'|=2,\\
    & \dim_R(R/\gam IR)=\dim_R(R)=d-|X|=d-2,\\
     &\ann_R (\h cIR)=(Z)/J=(u_5,\dots, u_l)/J,\\
    & \dim_R(R/ \ann_R (\h cIR))=d-|Z|=d-l+4<d-2.
     \end{align*}
    In particular,  Lynch's conjecture does not hold in this case. Note that Bahmanpour, in \cite[Example 3.2]{b2017}, obtained  $\dim_R(R/ \ann_R (\h cIR))$ without  computing $\ann_R (\h cIR)$.

    (ii) (Singh--Walther's example \cite{sw}). Let $K$ be a field and $S:=K[x, y, z_1, z_2]$ (or $S:=K[[x, y, z_1, z_2]]$). If we set $U:=\{x, y, z_1, z_2\}$, $X=X':=\{x\}$, $Y=Y':=\{y\}$, $Z:=\{z_1, z_2\}$, $J:=(X)\cap(Y)\cap(Z)=(xyz_1, xyz_2)$, $R:=S/J$ and $I:=(X')+(Y')+J/J=(x, y)/J$, then Example \ref{exa} implies that
        \begin{align*}
    & c:=\cdd R IR=|X'|+|Y'|=2,\\
    & \dim_R(R/\gam IR)=\dim_R(R)=d-|X|=4-1=3,\\
     &\ann_R (\h cIR)=(Z)/J=(z_1, z_2)/J,\\
    & \dim_R(R/ \ann_R (\h cIR))=d-|Z|=4-2=2.
     \end{align*}
    It follows that Lynch's conjecture is false.  This example also shows that \cite[Proposition 4.3 and Theorem 4.4]{l} are not true.
\end{rem}

We recall that an $R$-module $N$ is called {\it minimax} if there exists a finitely generated submodule $L$ of $N$ such that $N/L$ is Artinian. The class of minimax modules  includes all Noetherian and all Artinian modules; see \cite{zo}. Also, for an ideal $\fa$ of $R$,  Hartshorne \cite{h} defined an $R$-module $N$  to be {\it $\fa$-cofinite} if $\supp_R(N)\subseteq\V(\fa)$ and $\ext iR{R/\fa}N$ is finitely generated for all $i\in\N_0$. The following lemmas are needed to prove our next theorem.

\begin{lem}[{\cite[Theorem 1.6]{mel1999}}]\label{cofart} Let $(R, \fn)$ be a complete local ring, $\fa$ a proper ideal of $R$ and  $N$  an Artinian $R$-module. Then the following conditions on $N$ are equivalent:
\begin{enumerate}[\rm(i)]
\item $N$ is $\fa$-cofinite.
\item $(0:_{N}\fa)$ is finitely generated
\item For each attached prime ideal $\fp$ of $N$, $\surd(\fa+\fp)=\fn$.
\end{enumerate}
\end{lem}
\begin{lem}[{\cite[Corollary 4.4]{mel}}]\label{ser} Let $\fa$ be an ideal of $R$. The class of  $\fa$-cofinite minimax $R$-modules is closed under taking submodules,
quotients and extensions, i.e., it is a Serre subcategory of the category of $R$-modules.
\end{lem}
 Let $(R, \fn)$ be a complete  local ring, $\fa$ an ideal of $R$,  $N$ a finitely generated $R$-module with $N\neq \fa N$ and $c:=\cdd R{\fa}N$. In \cite[Theorem 2.4]{nr}, Rastgoo and Nazari proved that if $\h c{\fa}N$ is $\fa$-cofinite Artinian, then $\att_R(\h c{\fa}N)=\{\fp\in\mass_R(N): \dim_R(R/\fp)=c, \surd(\fa+\fp)=\fn\}$. The following lemma generalizes this theorem.
\begin{lem}\label{annh5} Let $(R, \fn)$ be a complete  local ring, $\fa$ an ideal of $R$,  $N$ a finitely generated $R$-module with $N\neq \fa N$ and $c:=\cdd R{\fa}N$. Suppose that  $\h c{\fa}N$ is Artinian and $(0:_{\h c{\fa}N}\fa)$ is finitely generated. Set $\Delta:=\{\fp\in\ass_R(N): \cdd R{\fa}{R/ \mathfrak {p}}=c\}$ and
$\Sigma:=\{\fp\in\mass_R(N): \dim_R(R/\fp)=c, \ \surd(\fa+\fp)=\fn\}$. Then
$$\matt_R(\h c{\fa}N)=\att_R(\h c{\fa}N)=\Delta=\Sigma,$$
 $$\textstyle{\T {\fa}N=\bigcap_{\fp\in\Sigma}C^N_\fp(0)},$$
 $$ \textstyle{\ann_R(\h c{\fa}N)=\ann_R(N/\T {\fa}N)}.$$
\end{lem}
\begin{proof}
Note that, by Lemma \ref{cofart}, $\h c{\fa}N$ is $\fa$-cofinite and Artinian. We prove the claimed equalities in some steps.

1) $\att_R(\h c{\fa}N)\subseteq\Delta$. Assume that $\fp\in\att_R(\h c{\fa}N)$. Hence by  Lemma \ref{lem3}, $\h c{\fa}{N/\fp N}\cong\h c{\fa}{N}/\fp\h c{\fa}N$.
     Thus $\h c{\fa}{N/\fp N}\neq 0$ because $\fp\in\att_R(\h c{\fa }{N})$.  Since $\att_R(\h c{\fa}N)\subseteq\V(\ann_R(N))$, we have $\supp_R(N/\fp N)=\supp_R(R/\fp)$ and so     $\cdd R{\fa}{R/\fp}=\cdd R{\fa}{N/\fp N}=c$. Therefore $\fp\in\Delta$.

2) $\Sigma\subseteq\Delta$.    If $\fp\in\Sigma$, then $\cdd R{\fa}{R/\fp}=\cdd R{\fa+\fp}{R/\fp}=\cdd R{\fn}{R/\fp}=\dim_R(R/\fp)=c$ and so $\fp\in\Delta$.

3) $\Delta\subseteq \Sigma$ and $\Delta\subseteq \matt_R(\h c{\fa}N)$. Assume that $\fp\in\Delta$
  and $L$ is an arbitrary $\fp$-primary submodule of $N$. Since $\ass_R(N/L)=\{\fp\}$, we obtain $\cdd R  {\fa}{N/L}=\cdd R {\fa}{R/\fp}=c$. Also,  it follows from the exact sequence $\h c{\fa}N\rightarrow\h c{\fa}{N/L}\rightarrow 0$ in view of Lemma \ref{ser} that $\h c{\fa}{N/L}$ is  $\fa$-cofinite  Artinian.  Now assume that $\fq\in\att_R(\h c{\fa}{N/L})$. Since
 $\att_R(\h c{\fa}{N/L})\subseteq\V({\ann_R (\h c{\fa}{N/L})})\subseteq\V(\ann_R (N/L)),$ we have $\fp\subseteq\fq$. By Lemma \ref{lem3}, $\h c{\fa}{N/\fq N}\cong\h c{\fa}{N}/\fq \h c{\fa}N$.
     Thus $\h c{\fa}{N/\fq N}\neq 0$ because $\fq\in\att_R(\h c{\fa }{N})$. Since $\supp_R(N/\fq N)=\supp_R(R/\fq)$, $\cdd R {\fa}{R/\fq}=\cdd R {\fa}{N/\fq N}=c$. Also, Lemma \ref{cofart} implies that $\surd(\fa+\fq)=\fn$ and hence, by Remark \ref{rem2}, $c=\cdd R {\fa}{R/\fq}=\cdd R {\fn}{R/\fq}=\dim_R{R/\fq}$. Therefore $\ann_R (\h c{\fa}N)\subseteq C_\fq(\ann_R (N))$ by Theorem \ref{annh2}.  We show that $\fq$ is minimal in $\supp_R(N)$. Assume that $\fq'\in{\rm Min}\supp_R(N)$ is such that $\fq'\subseteq\fq$. Then, by Lemma \ref{cont}(iii),
  $\ann_R (\h c{\fa}N)\subseteq\fq'$. Thus $\fq'\in{\rm Min}\V(\ann_R (\h c{\fa}N))={\rm Min}\att_R(\h c{\fa}N)$ and so  $\surd(\fa+\fq')=\fn$ by Lemma \ref{cofart}. Hence $\cdd R {\fa}{R/\fq'}=\cdd R {\fn}{R/\fq'}=\dim_R{R/\fq'}$. Also, we have $c=\cdd R {\fa}{R/\fq}\leq \cdd R {\fa}{R/\fq'}\leq \cdd R {\fa}{N}=c$. It follows that $\cdd R {\fa}{R/\fq'}=\dim_R{R/\fq'}=c$. Therefore  $\fq=\fq'$ and consequently  $\fq$ is a minimal element of $\supp_R(N)$. Since $\fp\subseteq\fq$, we have $\fp=\fq$. The equalities $\fp=\fq=\fq'$ show that $\fp\in\matt_R(\h c{\fa}N)$ and $\fp\in\Sigma$. Therefore $\Delta\subseteq \Sigma$ and $\Delta\subseteq \matt_R(\h c{\fa}N)$.

  (1), (2) and (3)  prove  that $\matt_R(\h c{\fa}N)=\att_R(\h c{\fa}N)=\Delta=\Sigma$. Finally, since $\Delta=\Sigma$,  Theorem \ref{annh2} implies that $\ann_R(\h c{\fa}N)=\ann_R(N/\T {\fa}N)$. Also it is clear that $\T {\fa}N=\bigcap_{\fp\in\Sigma}C^N_\fp(0)$ (see Definition \ref{def1}).  This completes the proof.
\end{proof}

\begin{thm}\label{annh4}
Let $N$ be a  finitely generated $R$-module, $\fa$ an ideal of $R$ such that $N\neq\fa N$ and $c:=\cdd R {\fa}N$. Let $0=N_1\cap\hdots\cap N_n$  be a minimal primary decomposition of the zero submodule of $N$ with $\ass_R(N/N_i):=\{\fp_i\}$ for all $1\leq i\leq n$. Set $\Delta:=\{\fp\in\ass_R(N): \cdd R {\fa}{R/\fp}=c\}$. If $(0:_{\h c{\fa}{N/N_i}}\fa)$ is finitely generated for all $\fp_i\in\Delta$ (or $(0:_{\h c{\fa}{N}/M}\fa)$ is finitely generated for all submodules $M$ of $\h c{\fa}N$), then
$$\ann_R (\h c{\fa}N)=\ann_R (N/\T {\fa}N).$$
In particular, if $N$ is coprimary and $(0:_{\h c{\fa}N}\fa)$ is finitely generated, then $$\ann_R (\h c{\fa}N)=\ann_R (N).$$
\end{thm}
\begin{proof}
Set $T:=\T {\fa}N=\bigcap_{\fp_i\in\Delta}N_i$. Since $\cdd R {\fa}T<c$,
$\h c{\fa}N\cong\h c{\fa}{N/T}$ and hence   $\ann_R (N/T)\subseteq\ann_R (\h c{\fa}{N/T})=\ann_R (\h c{\fa}N)$. Now we show that  $$\ann_R (\h c{\fa}N)\subseteq\ann_R (N/T).$$

     To prove the claimed inclusion, we assume that $r\in R$, $r\notin \ann_R (N/T)$ and it is sufficient for us  to show that $r\notin\ann_R(\h c{\fa}N)$. Since $rN\nsubseteq T$, $rN\nsubseteq N_i$ for some $N_i$ with $\fp_i\in\Delta$. As $\ass_R (r(N/N_i))=\ass_R(N/N_i)=\{\fp_i\}$, we see $\supp_R (r(N/N_i))=\supp_R(N/N_i)=\supp_R(R/\fp_i)$ and so
 $\cdd R {\fa}{r(N/N_i)}=\cdd R {\fa}{N/N_i}=c.$
Now suppose that $\fn\in{\rm Min}\supp_R(\h c{\fa}{N/N_i})$. It follows from $\supp_{R_\fn} ((r(N/N_i))_\fn)=\supp_{R_\fn}((N/N_i)_\fn)=\V(\fp_iR_\fn)$ that
$\cdd R {\fa R_\fn}{r/1(N/N_i)_\fn}=\cdd R {\fa R_\fn}{(N/N_i)_\fn}=c$. Therefore   $\h c{\fa R_\fn}{r/1(N/N_i)_\fn}\neq 0$ and hence
$$\h c{\fa \widehat{R_\fn}}{(r/1) \hat\,((N/N_i)_\fn)\, \hat{}\, }\neq 0,$$ where " $\hat{}$ " denotes the $\fn R_\fn$-adic completion.

  As $(0:_{\h c{\fa}{N/N_i}}\fa)$ is a finitely generated $R$-module, in view of \cite[Theorem 7.11]{mat},  $(0:_{\h c{\fa \widehat{R_\fn}}{((N/N_i)_\fn)\, \hat{}\,}}\fa\widehat{R_\fn})$ is a finitely generated  $\widehat{R_\fn}$-module. Also, by \cite[Theorem 4.3.2]{bs}, \cite[Theorem 23.2(ii)]{mat} and these facts that $\ass_{R_\fn}(\h c{\fa R_\fn}{(N/N_i)_\fn}=\{\fn R_\fn\}$ and $\widehat{R_\fn}$ is a local ring with maximal ideal $\fn\widehat{R_\fn}$,  we have
\begin{align*}
&\ass_{\widehat{R_\fn}}(0:_{\h c{\fa \widehat{R_\fn}}{((N/N_i)_\fn)\, \hat{}\,}}\fa\widehat{R_\fn})=\V(\fa\widehat{R_\fn})\cap\ass_{\widehat{R_\fn}}({\h c{\fa \widehat{R_\fn}}{((N/N_i)_\fn)\, \hat{}\,}})\\
 &= \ass_{\widehat{R_\fn}}({\h c{\fa \widehat{R_\fn}}{((N/N_i)_\fn)\, \hat{}\,}})
 =\ass_{\widehat{R_\fn}}({\h c{\fa {R_\fn}}{{(N/N_i)_\fn}}}\otimes_{\widehat{R_\fn}}\widehat{R_\fn})\\
 &=\textstyle{\bigcup_{\fp R_\fn\in\ass_{R_\fn}(\h c{\fa R_\fn}{(N/N_i)_\fn})}\ass_{\widehat{R_\fn}}(\widehat{R_\fn}/\fp \widehat{R_\fn})}
 =\ass_{\widehat{R_\fn}}(\widehat{R_\fn}/\fn \widehat{R_\fn})=\{\fn \widehat{R_\fn}\}.
 \end{align*}
It follows that $(0:_{\h c{\fa \widehat{R_\fn}}{((N/N_i)_\fn)\, \hat{}\,}}\fa\widehat{R_\fn})$ has finite length and so ${\h c{\fa \widehat{R_\fn}}{((N/N_i)_\fn)\, \hat{}\,}}$ is an Artinian $\widehat{R_\fn}$-module by Melkersson's Theorem (see \cite[Theorem 7.1.2]{bs}).

  Now,  $\h c{\fa \widehat{R_\fn}}{(r/1) \hat\,((N/N_i)_\fn)\, \hat{}\, }\neq 0$ yields $(r/1) \hat\,((N/N_i)_\fn)\, \hat{}\, \nsubseteq {\rm T}(\fa \widehat{R_\fn}, ((N/N_i)_\fn)\, \hat{} \,)$ because $\T{\fa \widehat{R_\fn}}{((N/N_i)_\fn)\, \hat{} \,}$  is the largest $\widehat{R_\fn}$-submodule $S$ of $((N/N_i)_\fn)\, \hat{}$
  such that $\h c{\fa \widehat{R_\fn}}{S}=0$.
  By Lemma \ref{annh5}, $$\ann_{\widehat {R_\fn}}(\h c{\fa \widehat{R_\fn}}{((N/N_i)_\fn)\, \hat{}\, })=\ann_{\widehat {R_\fn}}(((N/N_i)_\fn)\, \hat{}/\T{\fa \widehat{R_\fn}}{((N/N_i)_\fn)\, \hat{} \,}).$$
  Thus  $(r/1) \hat\,\h c{\fa \widehat{R_\fn}}{((N/N_i)_\fn)\, \hat{}\, }\neq 0$ and so $r\h c{\fa}{N/N_i}\neq 0$.  Next it follows from the exact sequence $\h c{\fa}N\rightarrow \h c {\fa}{N/N_i}\rightarrow 0$ that $r\h c{\fa}N\neq 0$, as required. This proves the first claimed equality. Note that if $(0:_{\h c{\fa}{N}/M}\fa)$ is finitely generated for all submodules $M$ of $\h c{\fa}N$, then the $R$-module $(0:_{\h c{\fa}{N/N_i}}\fa)$ is finitely generated for all $1\leq i\leq n$ because $\h c{\fa}{N/N_i}$ is a homomorphic image  of $\h c{\fa}N$. Finally assume that $N$ is coprimary and  $\ass_R(N)=\{\fp\}$.  Then $0$ is a $\fp$-primary submodule of $N$ and so $\T {\fa}N=0$.  Therefore, by the first part of theorem, we have $\ann_R(\h c{\fa}N)=\ann_R(N)$.
\end{proof}

\begin{cor}\label{cor2} Let $N$ be a non-zero finitely generated $R$-module, $\fa$ an ideal of $R$ with $N\neq\fa N$ and $c:=\cdd R{\fa}N$. Suppose that one of the following conditions holds:

\begin{enumerate}[\rm(i)]
\item $c\leq 1$;
\item $\dim_R(N/\fa N)\leq 1$;
\item $\dim_R(N)\leq 2$;
\item $\h {c}{\fa}N$ is $\fa$-cofinite minimax.
%\item $\h j{\fa}N$ is minimax for all $j<c$.
\end{enumerate}
Then
$$\ann_R (\h {c}{\fa}N)=\ann_R (N/\T {\fa}N),$$
 $$\Ht N{\ann_R(\h {c}{\fa}N)}=0,$$
and
\begin{align*}&\dim_R(R/\ann_R (\h {c}{\fa}N))\\
&=\max\{\dim_R(R/\fp): {\fp\in\mass_R(N),}\ {\cdd R {\fa}{R/\fp}=c}\}
\end{align*}
\end{cor}
\begin{proof}
Set $\Delta:=\{\fp\in\ass_R(N): \cdd R{\fa}{R/\fp}=c\}$. We have $\ann_R(N/\T{\fa}N)=\ann_R(N/\bigcap_{\fp\in\Delta}C^N_{\fp}(0))=\bigcap_{\fp\in\Delta}C_\fp(\ann_R(N))$. Since $\Delta\neq\emptyset$, it follows from Lemma \ref{cont}(iii) that $$\Ht N{\ann_R(N/\T{\fa}N)}=0.$$ Also,
 by Proposition \ref{prop5}(i), we have
 $$\dim_R(R/\ann_R(N/\T {\fa}N))=\max_{\fp\in\ass_R(N/\T {\fa}N)}\dim_R(R/\fp)=\max_{\fp\in\Delta}\ \dim_R(R/\fp).$$
 Now if $\fp\in\Delta$ and $\fq\in\mass_R(N)$ is such that $\fq\subseteq\fp$, then $c=\cdd R{\fa}{R/\fp}\leq\cdd R{\fa}{R/\fq}\leq c$ and so $\fq\in\Delta$. Therefore
  $$\dim_R(R/\ann_R(N/\T {\fa}N))=\max_{\fp\in\mass_R(N), \ \cdd R{\fa}{R/\fp}=c}\dim_R(R/\fp).$$
  Thus, to complete our proof,  we only need to prove that the first claimed equality $\ann_R(\h c{\fa}N)=\ann_R(N/\T {\fa}N)$ holds in all the given cases.

  Suppose that $0=N_1\cap\hdots\cap N_n$  is a minimal primary decomposition of the zero submodule of $N$ with $\ass_R(N/N_j):=\{\fp_j\}$ for all $1\leq j\leq n$. Assume that $\fp_i$ is an arbitrary element of $\Delta$. We claim that in  all the given cases either $\h c{\fa}{N/N_i}$ is $\fa$-cofinite or $(0:_{\h c{\fa}{N/N_i}}\fa)$ is finitely generated and hence we can deduce from Theorem \ref{annh4} that  the equality $\ann_R (\h c{\fa}N)=\ann_R (N/\T {\fa}N)$  holds.

(i) If  $\cdd R {\fa}N\leq 1$, then $\cdd R {\fa}{N/N_i}\leq 1$ and so, by \cite[Corollary 3.14]{mel}, $\h c{\fa}{N/N_i}$ is $\fa$-cofinite.

 (ii) If $\dim_R(N/\fa N)\leq 1$, then $\dim_R((N/N_i)/\fa(N/N_i))\leq 1$    and so $\h c{\fa}{N/N_i}$ is $\fa$-cofinite  by \cite[Corollary 2.7]{bn} (see also \cite[Theorem 1]{dm} and \cite[Theorem 1.1]{y} for the local case).

(iii) If $c=\dim_R(N/N_i)$, then $\h c{\fa}{N/N_i}$ is $\fa$-cofinite by \cite[Proposition 5.1]{mel} (see also \cite[Theorem 3]{dm} for the local case). If $c<\dim_R(N/N_i)$, then the result follows from (i).

 (iv) It follows from the exact sequence $\h c{\fa}N\rightarrow \h c{\fa}{N/N_i}\rightarrow 0$ that $\h c{\fa}{N/N_i}$ is also $\fa$-cofinite minimax because the category of $\fa$-cofinite minimax modules is a Serre subcategory of the category of $R$-modules; see Lemma \ref{ser}.
 %(v) In this case, by \cite[Theorem 2.8]{ma}, $\h j{\fa}N$ is $\fa$-cofinite minimax for all $j<c$. Therefore \cite[Theorem 2.2]{f2013} implies that  $(0:_{\h c{\fa}{N/N_i}}\fa)$ is $\fa$-cofinite minimax. In particular,
% $(0:_{(0:_{\h c{\fa}{N/N_i}}\fa)}\fa)=(0:_{\h c{\fa}{N/N_i}}\fa)$ is finitely generated.
\end{proof}
%Let $R$ be a local ring, $x_1,\dots, x_t\in R$, $\fa:=(x_1,\dots, x_t)$ and $N$  a  finitely generated $R$-module.   Schenzel  proved that $(\ann_R (\h 0{\fa}N)\dots(\ann_R (\h t{\fa}N)\subseteq\ann_R (N)$; see   \cite[page 350]{sch}.  Atazadeh et al., in \cite[Theorem 2.11]{asn2017}, used this formula to proved that if $N$ is relative Cohen-Macaulay with respect to $\fa$, then $\ann_R(\h c{\fa}N)=\ann_R(N)$. The following corollary generalizes this result to the non-local case. Note that  when $N$ is relative Cohen-Macaulay, we have $N\neq\fa N$ and so $\fa$ is a proper ideal of $R$. Therefore if  $R$ is local, then for each $a\in\fa$, $1-a$ is a unit of $R$ and so is a non-zerodivisor on $N$.
%\begin{cor} \label{rel2}Let $N$ be a finitely generated $R$-module and  $\fa$  an ideal of $R$. Let  $N$ be a relative Cohen-Macaulay with respect to $\fa$ and $c:=\grad{R} {\fa}N=\cdd R{\fa}N$. Then $\ann_R(\h c{\fa}N)=\ann_R(N/{\rm S}(\fa, N))$. In particular,
%If for each $a\in\fa$, $1-a$ is a non-zerodivisor on $N$,  then  $\ann_R(\h c{\fa}N)=\ann_R(N).$
%\end{cor}
% \begin{proof}
% Since $N$ is relative Cohen-Macaulay with respect to $\fa$, ${\rm S}(\fa, N)=\T {\fa}N$.
 %Hence, by Corollary \ref{cor2}(v), $\ann_R(\h c{\fa}N)=\ann_R(N/{\rm S}(\fa, N))$.
 % Finally if, for each $a\in\fa$, $1-a$ is a non-zerodivisor on $N$, then  it follows from Proposition \ref{prop5}(iv) that  ${\rm S}({\fa}, N)=0$. This completes the proof.
 %\end{proof}
   Bahmanpour in \cite[Theorem 2.9]{b2015}, as one of his main results,   computed the annihilator of $\h c{\fa}N$ when $c:=\cdd R{\fa}N=1$. The following corollary shows that our result in Corollary \ref{cor2}(i) coincides with that of Bahmanpour and gives an affirmative answer to Lynch's conjecture.
   %The last part of the following proposition is needed. This proposition characterizes the finitely generated $R$-modules of cohomological dimension zero.
\begin{lem}\label{prop2} Let $N$ be  a non-zero finitely generated (coprimary) $R$-module with $\ass_R(N)=\{\fp\}$ and $\fa$ an ideal of $R$.
%\begin{enumerate}[\rm(i)]
%\item $\cdd R{\fa}N=0$ if and only if $\fa^tN=\fa^{t+1}N$ for some $t\in\N$.
%\item If for each $a\in\fa$, $1-a$ is a non-zerodivisor on $N$, then $\cdd R{\fa}N=0$ if and only if $\gam {\fa}N=N$.
%\item
 Then  $\cdd R{\fa}N=0$ if and only if $\fa\subseteq\fp$.
 %\end{enumerate}
\end{lem}
\begin{proof}
%(i) Assume that $\cdd R{\fa}N=0$. Corollary \ref{rel2} implies that $\ann_R(\gam {\fa}N)=\ann_R(N/{\rm S}(\fa, N))$. There exists $t\in\N$ such that $\fa^t\gam {\fa}N=0$. Hence $\fa^tN\subseteq{\rm S}(\fa, N)$ and so   $\fa^tN=\fa(\fa^{t}N)$ by Proposition \ref{prop5}(v).
% Conversely, assume that $\fa^tN=\fa(\fa^tN)$ for some $t\in\N$. Then $\h i{\fa}{\fa^tN}=0$ for all $i\in\N_0$ and so $\h i{\fa}N\cong\h i{\fa}{N/\fa^tN}$ for all $i\in\N_0$. It follows that $\cdd R{\fa}N=0$.

%(ii) Assume that for each $a\in\fa$, $1-a$ is a non-zerodivisor on $N$. Therefore ${\rm S}(\fa, N)=0$. On the other hand, for $t\in\N$, $\fa^tN=\fa(\fa^tN)$ if and only if
%$\fa^tN\subseteq{\rm S}(\fa, N)$, see Proposition \ref{prop5}(v). Now the assertion follows from (i).

%(iii)
% Let $\fp$ be a  prime ideal of $R$,  $N$   a non-zero finitely generated $\fp$-primary $R$-module and $\fa'$  an ideal of $R$.
If  $\fa\subseteq \fp$, then, since $\surd(\ann_R(N))=\fp$, we obtain $\fa^nN=0$ for some $n\in\N$ and so $\gam {\fa}N=N\neq 0$. It follows that $\h i{\fa}N=0$ for all $i>0$ and consequently $\cdd R{\fa}N=0$. Conversely, assume that $\cdd R{\fa}N=0$. Hence $\gam {\fa}N\neq 0$. Thus $\ass_R(\gam {\fa}N)=\V {(\fa)}\cap\ass_R(N)\neq \emptyset$ and so $\fa\subseteq\fp$.
\end{proof}
\begin{cor} Let $\fa$ be an ideal of $R$ and $N$ a finitely generated $R$-module such that $\cdd R{\fa}N=1$. Then $\ann_R(\h 1{\fa}N)=\ann_R(N/\gam {\fb}N),$ where $\fb:=\bigcap_{\fp\in\Delta}\fp$ and $\Delta:=\{\fp\in\ass_R(N): \fa+\fp=R \textrm{ or } \fa\subseteq\fp\}$.
In particular, if for each $a\in\fa$, $1-a$ is a non-zerodivisor on $N$, then $\ann_R(\h 1{\fa}N)=\ann_R(N/\gam {\fa}N)$.
\end{cor}
\begin{proof}
By Corollary \ref{cor2}(i), $\ann_R(\h 1{\fa}N)=\ann_R(N/\T{\fa}N)$. Also, by Definition \ref{def1}, $\T{\fa}N=\gam {\fb'}N$, where $\fb'=\bigcap_{\fp\in\Delta'}\fp$ and $\Delta'=\{\fp\in\ass_R(N): \cdd R{\fa}{R/\fp}<1\}$.
Now assume that $\fp\in\ass_R(N)$. Then $\cdd R{\fa}{R/\fp}=-\infty$ if and only if $\fa(R/\fp)=R/\fp$ or equivalently $\fa+\fp=R$. Also, by Lemma \ref{prop2}, $\cdd R{\fa}{R/\fp}=0$ if and only if $\fa\subseteq\fp$. Therefore $\Delta'=\Delta$ and so $\fb=\fb'$. This proves the first part of the assertion.
 Now assume that for each $a\in\fa$, $1-a$ is a non-zerodivisor on $N$. Thus $\Delta=\{\fp\in\ass_R(N): \fa\subseteq\fp\}=\ass_R(\gam {\fa}N)$ and so
 $\fb=\surd(\ann_R(\gam {\fa}N))$. It follows that $\gam {\fb}N=\gam {\ann_R(\gam {\fa}N)}N=\gam {\fa}N$ because $\fa^t\gam {\fa}N=0$ for some $t\in\N_0$.
\end{proof}
Let $\fa$ be a proper ideal of $R$. Then, by  \cite[Corollary 3.3.3]{bs}, $\ara (\fa)$ is an upper bound for the invariant  $\cdd R{\fa}R$.
Hochster and Jeffries \cite[Theorem 2.6]{hj} proved that if $R$ is a domain of prime characteristic and $\cdd R{\fa}R=\ara (\fa)$, then $\h {\cdd R{\fa}R}{\fa}R$ is faithful. On the other hand, if $R$ is local, then  $\dim_R(R)-\dim_R(R/\fa)$ is a lower bound for  the invariant $\cdd R{\fa}R$; see the following lemma. If $R$ is a local domain and $\cdd R{\fa}R=\dim_R(R)-\dim_R(R/\fa)$, then   $\h {\cdd R{\fa}R}{\fa}R$ is faithful, see \cite[Theorem 2.7]{b2015}. In the following theorem we generalize this result.
We note that Theorem \ref{cd2} and Corollary \ref{sys} give  affirmative answers to Lynch's conjecture.

\begin{lem}[{\cite[Corrollary 2.3]{ey}}]\label{erd} Let $(R, \fn)$ be a local ring, $\fa$ an ideal of $R$ and $N$ a finitely generated $R$-module with $N\neq \fa N$. Then
$$\dim_R(N)-\dim_R(N/\fa N)\leq \cdd R{\fa}N.$$
 Moreover, if $\cdd R{\fa}N=\dim_R(N)-\dim_R(N/\fa N)$, then $$\h {\dim_R(N/\fa N)}{\fn}{\h {\cdd R{\fa}N}{\fa}N}\cong \h {\dim_R(N)}{\fn}N$$
and $\dim\supp_R({\h {\cdd R{\fa}N}{\fa}N})=\dim_R(N/\fa N)$.
\end{lem}

\begin{thm}\label{cd2} Let $(R, \fn)$ be a local ring, $\fa$ an ideal of $R$, $N$ a finitely generated $R$-module with $N\neq \fa N$ and $c:=\cdd R{\fa}N=\dim_R(N)-\dim_R(N/\fa N)$. Then \begin{enumerate}[\rm(i)]
\item $\ann_R(\h c{\fa}N)\subseteq\ann_R(N/\bigcap_{\fp\in\assh_R(N)}C^N_\fp(0)).$
 \item If   $N$ is  unmixed (that is, $\dim_R(N)=\dim_R(R/\fp)$ for all $\fp\in\ass_R(N)$), then $\ann_R(\h {c}{\fa}N)=\ann_R(N)$.
\item $\Ht N{\ann_R(\h c{\fa}N)}=0$.
\item $\dim_R(R/\ann_R(\h c{\fa}N))=\dim_R(N)$. Moreover, if  $c>0$, then   $\dim_R(N)=\dim_R(N/\gam {\fa}N)=\dim_R(R/\ann_R(\h c{\fa}N)$.
\end{enumerate}
\end{thm}
\begin{proof}
(i) By Lemma \ref{erd}, we have $\h {\dim_R(N/\fa N )}{\fn}{\h c{\fa}N}\cong\h {\dim_R(N)}{\fn}N$. Thus
$$\ann_R(\h c{\fa}N)\subseteq\ann_R(\h {\dim_R(N)}{\fn}N).$$
Now (i) follows from Theorem \ref{annh2} (set $\fa=\fn$ in Theorem \ref{annh2}).

(ii) Assume that $N$ is unmixed. Then   we have $\assh_R(N)=\ass_R(N)$ and hence $\bigcap_{\fp\in\assh_R(N)}C^N_\fp(0)=\bigcap_{\fp\in\ass_R(N)}C^N_\fp(0)=0$. Therefore, by (i),   $\ann_R(\h c{\fa}N)\subseteq\ann_R(N)$. The reverse inclusion is clear and so the claimed equality holds.

 (iii) It is an immediate consequence of (i).

(iv) Let $\fp\in\assh_R(N)$. By (i),   $\ann_R(N)\subseteq\ann_R(\h c{\fa}N)\subseteq C_\fp(\ann_R(N))\subseteq \fp$ and so
$$\dim_R(R/\fp)\leq\dim_R(R/\ann_R(\h c{\fa}N))\leq\dim_R(R/\ann_R(N)).$$
Therefore $\dim_R(R/\ann_R(\h c{\fa}N))=\dim_R(N)$. Now assume that $c>0$. Hence
$$\ann_R(N)\subseteq\ann_R(N/\gam {\fa}N)\subseteq\ann_R(\h c{\fa}{N/\gam {\fa}N})=\ann_R(\h c{\fa}N)\subseteq\fp$$
and so the claimed equalities hold.
\end{proof}
 The following corollary  generalizes two main theorems of \cite{b2015} (see \cite[Theorems 2.2 and 2.3]{b2015}).
\begin{cor}\label{sys} Let $(R, \fn)$ be a local ring and  $N$ a non-zero finitely generated $R$-module. Let $0\leq t\leq \dim_R(N)$ and $x_1,\dots,x_t\in\fn$ be a part of a system of parameters for $N$. Then the following statements hold.
\begin{enumerate}[\rm(i)]
\item $\cdd R{(x_1,\dots,x_t)}N=t$.
\item $\ann_R(\h t{(x_1,\dots,x_t)}N)\subseteq\ann_R(N/\bigcap_{\fp\in\assh_R(N)}C_\fp^N(0))$. In particular, if $N$ is unmixed, then $\ann_R(\h t{(x_1,\dots,x_t)}N)=\ann_R(N)$.
\item $\Ht N{\ann_R(\h t{(x_1,\dots,x_t)}N)}=0$.
\item $\dim_R(R/\ann_R(\h t{(x_1,\dots,x_t)}N))=\dim_R(N)$ and if $t>0$, then $$\dim_R(R/\ann_R(\h t{(x_1,\dots,x_t)}N))=\dim_R(N/\gam {(x_1,\dots,x_t)}N)=\dim_R(N).$$
    \item $\dim\supp_R(\h t{(x_1,\dots,x_t)}N))=\dim_R(N)-t$.
\end{enumerate}
\end{cor}
\begin{proof}
 We have $\dim_R(N/(x_1,\dots, x_t)N)=\dim_R(N)-t$ because $x_1,\dots, x_t$ is a part of a system of parameters of $N$. Hence, by  Lemma \ref{erd} and \cite[Corollary 3.3.3]{bs}, we obtain
\begin{align*}
&t=\dim_R(N)-\dim_R(N/(x_1,\dots, x_t)N)\\
&\leq\cdd R{(x_1,\dots,x_t)}N\leq\ara{(x_1,\dots,x_t)}\leq t.
\end{align*}
It follows that $$\cdd R{(x_1,\dots,x_t)}N=\dim_R(N)-\dim_R(N/(x_1,\dots, x_t)N)=t.$$
Therefore (i) holds; and (ii)--(iv) follow from Theorem \ref{cd2}. Finally, by Lemma \ref{erd}, we have $$\dim\supp_R(\h t{(x_1,\dots,x_t)}N))=\dim_R(N/(x_1,\dots,x_t)N)=\dim_R(N)-t.$$
\end{proof}
\section{\bf Annihilator of first non-zero local cohomology}
Let $\fa$ be an ideal of $R$ and $N$ a finitely generated $R$-module with $N\neq\fa N$. Our purpose in this section is to provide a sharp upper bound for the annihilator of
$\h {\grad{R} {\fa}N}{\fa}N$, see Theorem \ref{annh7} and Corollary \ref{cor3}. Also, we consider and compute $\dim_R(R/\ann_R(\h {\grad{R}{\fa}N}{\fa}N)$ in certain cases. Before that we need some lemmas.
\begin{lem}\label{loc} Let $(R, \fn)$ be a local ring and  $N$ a non-zero finitely generated  $R$-module. For each  $t\in\N_0$, set $\Delta(t):=\{\fp\in\ass_R(N): \dim_R(R/\fp)\geq t\}$,
$\Sigma(t):=\{\fp\in\mass_R(N): \dim_R(R/\fp)=t\}$ and $\Sigma'(t):=\{\fp\in\ass_R(N)\setminus\mass_R(N): \dim_R(R/\fp)=t\}$. There is the following bound for the annihilator of $\h t{\fn}N$:
\begin{align*}
&\textstyle{\ann_R(N/\bigcap_{\fp\in\Delta(t)}C_\fp^N(0))\subseteq\ann_R(\h t{\fn}N)}\\
&\textstyle{\subseteq\ann_R(N/\bigcap_{\fp\in\Sigma(t)}C_\fp^N(0))\cap(\bigcap_{\fp\in\Sigma'(t)}\fp).}
\end{align*}
\end{lem}
\begin{proof}
We set $S(t):=R\setminus\bigcup_{\fp\in\Delta(t)}\fp$ and $T(t):=R\setminus\bigcup_{\fp\in\Sigma(t)}\fp$. By \cite[Theorem 3.2]{f}, we have
$$\ann_R(N/C_{S(t)}^N(0))\subseteq\ann_R(\h t{\fn}N)\subseteq\ann_R(N/C_{T(t)}^N(0)).$$
Also, if $\fp\in\Sigma'(t)$, then \cite[Corollary 4.9]{s} implies that $\fp\in\att_R(\h t{\fn}N$ and so $\ann_R(\h t{\fn}N)\subseteq\fp$. This completes the proof.
\end{proof}
Note that if $N$ is a non-zero finitely generated $R$-module and $\fp\in\supp_R(N)$, then  we have $\ann_R(N/C^N_\fp(0))=C_\fp(\ann_R(N))\subseteq\fp$. Example \ref{exa2} shows that to improve the upper bound for the annihilator of local cohomology $\h t{\fn}N$ in the above lemma we can not replace  $\mass_R(N)$ by $\ass_R(N)$ in the   index set $\Sigma(t)$.
\begin{lem}\label{annh6}
Let $N$ be a finitely generated $R$-module and $\fa$ an ideal of $R$ with $N\neq\fa N$. Let $\fp\in\ass_R(N)$  with $\fa+\fp\neq R$ and  $n:=\Ht {R/\fp}{(\fa+\fp)/\fp}$. Then
$$\ann_R(\h n{\fa+\fp}N)\subseteq\fp.$$ If, in addition, $\fp\in\mass_R(N)$, then
$$\ann_R(\h n{\fa+\fp}N)\subseteq\ann_R(N/C^N_\fp(0)).$$
In particular, if $N$ is coprimary, then $\ann_R(\h {\Ht N{\fa}}{\fa}N)=\ann_R(N).$
\end{lem}
\begin{proof}
Assume that $\fq$ is a prime ideal of $R$ containing $\fa+\fp$ such that $\Ht{R/\fp}{\fq/\fp}=\Ht {R/\fp}{(\fa+\fp)/\fp}$. Then
$$\dim_{R_\fq}(R_\fq/\fp R_\fq)=\Ht {R_\fq/\fp R_\fq}{\fq R_\fq/\fp R_\fq}=\Ht {R/\fp}{\fq/\fp}=n.$$
Also  we have
$$(\h n{\fa+\fp}N)_\fq\cong \h n{(\fa+\fp)R_\fq}{N_\fq}\cong\h n{\fq R_\fq}{N_\fq}.$$
Since $\fp R_\fq\in\ass_{R_\fq}(N_\fq)$ and $\dim_{R_\fq}(R_\fq/\fp R_\fq)=n$, we have $\fp R_\fq\in\att_{R_\fq}(\h n{\fq R_\fq}{N_\fq})$ by \cite[Corollary 4.9]{s} and so
$\ann_{R_\fq}(\h n{\fq R_\fq}{N_\fq})\subseteq\fp R_\fq$. To prove $\ann_R(\h n{\fa+\fp}N)\subseteq\fp$, assume for the sake of contradiction that
$x\in\ann_R(\h n{\fa+\fp}N)$ and $x\notin\fp$. Since $\fp$ is prime, $x/1\notin\fp R_\fq$ and so $(x\h n{\fa+\fp}N)_\fq\cong x/1\h n{\fq R_\fq}{N_\fq}\neq 0$, a contradiction. This proves the first claimed inclusion.

Next, assume in addition that $\fp$ is a minimal element of $\ass_R(N)$. Thus $C:=C^N_\fp(0)$ is a $\fp$-primary submodule of $N$ (in fact, $C^N_\fp(0)$ is the unique $\fp$-primary component of every minimal primary decomposition of the zero submodule of $N$).
Therefore $\ass_R(N/C)=\{\fp\}$.
Since $\fp R_{\fq}$ is a minimal element of $\ass_{R_\fq}(N_\fq)$ with $\dim_{R_\fq}({R_\fq}/{\fp R_\fq})=n$, it follows from Lemma \ref{loc} that
$$\ann_{R_\fq}((\h n{\fa+\fp}N)_\fq)=\ann_{R_\fq}(\h n{\fq R_\fq}{N_\fq})\subseteq \ann_{R_\fq}(N_\fq/C^{N_\fq}_{\fp R_\fq}(0)).$$
Now suppose that $x\in R$ is such that $x N\nsubseteq C$. We have $\emptyset\neq\ass_R(x(N/C))\subseteq\ass_R(N/C)=\{\fp\}$ and so $\ass_R(x(N/C))=\{\fp\}$. Therefore
$\ass_{R_\fq}(x/1(N_\fq/C_\fq))=\{\fp R_\fq\}$ and hence $x/1\notin \ann_{R_\fq}(N_\fq/C_\fq)$. By Lemma \ref{lem5}, $C_\fq=(C^N_\fp(0))_\fq=C^{N_\fq}_{\fp R_\fq}(0)$.
 Hence the above inclusion implies that $x/1(\h n{\fa+\fp}N)_\fq\neq 0$ and so $x \h n{\fa+\fp}N\neq 0$. It follows that
$$\ann_R(\h n{\fa+\fp}N)\subseteq\ann_R(N/C)$$
because $x$ is an arbitrary element of $R$ with $x N\nsubseteq C$. The last assertion follows immediately from the first part, because if $\ass_R(N)=\{\fp\}$, then $C^N_\fp(0)=0$, $\Ht N{\fa}=\Ht {R/\fp}{(\fa+\fp)/\fp}$ and, since $\surd(\ann_R(N))=\fp$,  $\h {\Ht N{\fa}}{\fa}N=\h {\Ht N{\fa}}{\fa+\ann_R(N)}N=\h {\Ht N{\fa}}{\fa+\fp}N$.
\end{proof}

\begin{lem}\label{iso}
Let $N$ be a finitely generated $R$-module, $\fa$ an ideal of $R$ such that $\fa N\neq N$ and $n:=\grad{R} {\fa}N$. Then, for each ideal $\fb$ of $R$, there is the following isomorphism:
$$\h n{\fa+\fb}N\cong\gam {\fb}{\h n{\fa}N}.$$
\end{lem}
\begin{proof}
 Note that $n$ is a finite non-negative integer because  $\fa N\neq N$. By \cite[Theorem 10.47]{r}, there is the following Grothendieck's third quadrant
spectral sequence
$$\textstyle{E_2^{p, q}=\h p{\fb}{\h q{\fa}N}\underset{p}{\Rightarrow} \h n{\fa+\fb}N,}$$
where $n=p+q$.
We set ${H}^n:=\h n{\fa+\fb}N$. Hence  there is a finite filtration
$$0=\Phi_{n+1}{H}^n\subseteq\Phi_{n}{H}^n\subseteq\cdots\subseteq\Phi_0{H}^n={H}^n$$
of submodules of ${H}^n$ such that $\Phi_{p}{H}^n/\Phi_{p+1}{H}^n\cong E^{p, q}_\infty$ for all  $0\leq p\leq n$.
If $p>0$, then $q<n$ and so $E_2^{p, q}=0$. Therefore   $E_{\infty}^{p, q}=0$ because $E_\infty^{\fp, \fq}$ is a subquotient of $E_2^{\fp, \fq}$. Thus
$\Phi_1{H}^n=\Phi_2{H}^n=\dots=\Phi_{n+1}{H}^n=0$ and $H^n=\Phi_0H^n\cong E_\infty^{0,n}$.

On the other hand, for each $r\geq 2$, there is the following exact sequence
 \begin{displaymath}
 \xymatrix {E^{-r,n+r-1}_r\ar[rr]^{\quad d^{-r,n+r-1}_r}&&E^{0, n}_r\ar[r]^{d^{0, n}_r\qquad}&E^{r, n-r+1}_r }.
 \end{displaymath}
  $E^{-r,n+r-1}_r$ is a subquotients of $E^{-r,n+r-1}_2$ and so is zero. Also, since $n-r+1<n$, $E_2^{r, n-r+1}$ and, consequently, $E_r^{r, n-r+1}$ are zero. Therefore
  $$E^{0, n}_{r+1}= \textrm{ker}\,d^{0, n}_r/ \textrm{im}\, d^{-r,n+r-1}_r\cong E^{0, n}_r$$
  for all $r\geq2$. It follows that $E_\infty^{0, n}\cong E_2^{0, n}$ and so
  $$\h n{\fa+\fb}N=H^n\cong E_\infty^{0, n}\cong E_2^{0, n}=\gam {\fb}{\h n{\fa}N}.$$
\end{proof}
Now we are ready to state and prove the main theorem of this section which provides a bound for the annihilator of the first non-zero local cohomology module. We recall that if $N$ is a finitely generated $R$-module and $\fa$ is an ideal of $R$ with $N\neq\fa N$, then  ${\rm S}(\fa, N)=\bigcap_{\fp\in\ass_R(N),\, \fa+\fp\neq R}C^N_\fp(0)=\gam {\bigcap_{\fp\in\ass_R(N),\, \fa+\fp=R}\fp}{N}$ is the largest submodule $L$ of $N$ such that $L=\fa L$;  also, ${\rm S}(\fa, N)=0$ if and only if for each $a\in\fa$, $1-a$ is a non-zerodivisor on $N$; see  Definition \ref{def1}, Lemma \ref{grad} and Proposition \ref{prop5}.
\begin{thm}\label{annh7}
Let $N$ be a finitely generated $R$-module, $\fa$ an ideal of $R$ with $N\neq \fa N$ and $n:=\grad{R} {\fa}N$.  Then
$$\textstyle{\ann_R(N/{\rm S}(\fa, N))\subseteq\ann_R(\h n{\fa}N)\subseteq\ann_R(N/\bigcap_{\fp\in\Sigma} C^N_\fp(0))\cap(\bigcap_{\fp\in\Sigma'}\fp)}$$
where   $\Sigma:=\{\fp\in\mass_R(N): \Ht {R/\fp}{(\fa+\fp)/\fp}=n\}$ and $\Sigma':=\{\fp\in\ass_R(N)\setminus\mass_R(N): \Ht {R/\fp}{(\fa+\fp)/\fp}=n\}$.
% Moreover, if any element of $\ass_R(\h n{\fa}N)$ is a minimal element of $\supp_R(N/\fa N)$, then
%$$\ann_R(\h n{\fa}N)=\ann_R(N/\bigcap_{\fp\in\Sigma} C^N_\fp(0)).$$
\end{thm}
\begin{proof}
The lower bound for the annihilator of $\h n{\fa}N$ follows from Theorem \ref{lower}(iii). Also, it follows from Lemmas \ref{iso} and \ref{annh6}  that
$$\ann_R(\h n{\fa}N)\subseteq\ann_R(\gam {\fp}{\h n{\fa}N})=\ann_R(\h n{\fa+\fp}N)\subseteq\ann_R(N/C^N_{\fp}(0))$$
for all $\fp\in \Sigma$. Similarly, $\ann_R(\h n{\fa}N)\subseteq\fq$  for all $\fq\in\Sigma'$. Hence
\begin{align*}
\ann_R(\h n{\fa}N)&\textstyle{\subseteq\bigcap_{\fp\in\Sigma}\ann_R(N/C^N_{\fp}(0))\cap(\bigcap_{\fq\in\Sigma'}\fq)}\\
&=\textstyle{\ann_R(N/\bigcap_{\fp\in\Sigma} C^N_\fp(0))\cap(\bigcap_{\fq\in\Sigma'}\fq).}
\end{align*}
 \end{proof}

\begin{cor}\label{cor3}
Let $N$ be a finitely generated Cohen-Macaulay $R$-module, $\fa$ an ideal of $R$ with $N\neq \fa N$, $n:=\grad{R} {\fa}N$ and $\Sigma:=\{\fp\in\ass_R(N): \Ht {R/\fp}{(\fa+\fp)/\fp}=n\}$. Then
$$\textstyle{\ann_R(\h n{\fa}N)=\ann_R(N/\bigcap_{\fp\in\Sigma} C^N_\fp(0)).}$$
 In particular, $\Ht N{\ann_R(\h n{\fa}N)}=0$ and $\dim_R(R/\ann_R(\h n{\fa}N))=\dim_R(N)$.
\end{cor}
\begin{proof}
 Assume that $x\in R$ and $x\h n{\fa}N\neq 0$. Hence $(x\h n{\fa}N)_\fq\cong x/1\h n{\fa R_\fq}{N_\fq}\neq 0$ for some $\fq\in\ass_R(\h n{\fa}N)$. By \cite[Theorem 2.1]{ctt}, $\depth_{R_\fq}(N_\fq)=\grad{R}{\fa}N=n$ and so $\Ht {N}\fq=\Ht N{\fa}$. Therefore $\fq$ is a minimal prime ideal of $\fa+\ann_R(N)$ and hence $
 \h n{\fa R_\fq}{N_\fq}\cong\h n{\fq R_\fq}{N_\fq}$. Since $\dim_{R_\fq}(N_\fq)=n$, by \cite[Theorem 3.2(iii)]{f}, $\ann_{R_\fq}(\h n{\fq R_\fq}{N_\fq})=\ann_{R_\fq}(N_\fq)$. Thus $(xN)_\fq\neq 0$ and so there exists $\fp R_\fq\in\ass_{R_\fq}(N_\fq)$ such that $(xN_\fq)_{\fp R_\fq}\cong(xN)_\fp\neq 0$. Since $N_\fq$ is Cohen-Macaulay, $\dim_{R_\fq}(R_\fq/\fp R_\fq)=\dim_{R_\fq}(N_\fq)=n$ and so Lemma \ref{grad} yields
 $$n=\grad{R} {\fa}N\leq \Ht {R/\fp}{\fa+\fp}\leq\Ht {R/\fp}{\fq/\fp}=\dim_{R_\fq}(R_\fq/\fp R_\fq)=n.$$
 Therefore $\fp\in\Sigma$ and hence $(\bigcap_{\fp'\in\Sigma} C^N_{\fp'}(0))_\fp\subseteq (C^N_\fp(0))_\fp=0$. It follows that $(x(N/\bigcap_{\fp'\in\Sigma} C^N_{\fp'}(0)))_\fp\neq 0$ and so $xN\nsubseteq\bigcap_{\fp'\in\Sigma} C^N_{\fp'}(0)$. Since $x$ is an arbitrary element of $R$ with $x\h n{\fa}N\neq 0$, we have
 $\ann_R(N/\bigcap_{\fp'\in\Sigma} C^N_{\fp'}(0))\subseteq\ann_R(\h n{\fa}N)$. The reverse inclusion follows from Theorem \ref{annh7}.
 Finally since $\h n{\fa}N\neq 0$, $\Sigma\neq\emptyset$. Assume that $\fp\in\Sigma$. Then $$\ann_R(N)\subseteq\ann_R(\h c{\fa}N)\subseteq \ann_R(N/C^N_\fp(0))=C_\fp(\ann_R(N))\subseteq\fp.$$
  Since $\Ht N{\fp}=0$ and $\dim_R(R/\fp)=\dim_R(N)$, the above inclusions imply
 the last assertion.
\end{proof}

%Now assume that $N$ is a Cohen-Macaulay $R$-module.  Let $\fp\in\ass_R(N)$ and $\fq, \fq'$ are prime ideals containing $\fa+\fp$ such that $\Ht N{\fa+\fp}=\Ht N{\fq'}$ and $\Ht {R/\fp}{(\fa+\fp)/\fp}=\Ht {R/\fp}{\fq/\fp}$. Then
%\begin{align*}
%\Ht {R/\fp}{(\fa+\fp)/\fp}&=\Ht {R/\fp}{\fq/\fp}=\dim_{R_\fq}(R_\fq/\fp R_\fq)\\
%&=\dim_{R_\fq}(N_\fq)=\Ht N\fq\geq\Ht N{\fa+\fp}=\Ht N{\fq'}\\&
%=\dim_{R_{\fq'}}(N_{\fq'})=\dim_{R_{\fq'}}(R_\fq/\fp R_{\fq'})=\Ht {R/\fp}{\fq'/\fp}\\&\geq \Ht {R/\fp}{(\fa+\fp)/\fp}.
%\end{align*}
%Therefore $\Ht N{\fa+\fp}=\Ht {R/\fp}{(\fa+\fp)/\fp}$ for all $\fp\in\ass_R(N)$. Hence the reverse of the claimed inclusion follows by \cite[Theorem 3.6]{f}.

\begin{ex} \label{exa2} Let $K$ be a field and $R:=K[[x, y]]$ be the ring of formal power series over $K$ in indeterminates   $x, y$.  Set $N:=R/(Rx^2+Rxy), N_1:=Rx/(Rx^2+Rxy)$ and  $N_2:=(Rx^2+Ry)/(Rx^2+Rxy)$. Then $0=N_1\cap N_2$ is a minimal primary decomposition of the zero submodule of $N$ with $\ass_R(N/N_1)=\{\fp:=Rx\}$ and $\ass_R(N/N_2)=\{\fn:=Rx+Ry\}$. Thus $\ass_R(N)=\{\fp, \fn\}$.

 We have $\ext 1RNR\cong R/\fp$, $\ext 2RNR\cong R/\fn$ and $\ext iRNR=0$ for all $i\neq 1, 2$; see  \cite[Example 2.7]{f}. Therefore
 the Grothendieck's Duality Theorem (see \cite[Theorem 11.2.5 or 11.2.8]{bs}) implies that
$$\gam {\fn}N\cong\Hom R{\ext 2RNR}{E_R(R/\fn)}\cong\Hom R{R/\fn}{E_R(R/\fn)}\cong R/\fn,$$
\begin{align*}
&\h 1{\fn}N\cong\Hom R{\ext 1RNR}{E_R(R/\fn)}\cong\Hom R{R/\fp}{E_R(R/\fn)}\\
&\cong E_{R/\fp}(R/\fn)\cong E_{K[[y]]}(K)\cong K[y^{-1}]
\end{align*}
($K[y^{-1}]$ is a $K[[y]]$-module by the convention that $y^r.y^{-s}$ is equal $y^{-(s-r)}$ when $s\geq r$ and zero otherwise) and $\h i{\fn}N=0$ for all $i\neq 0, 1$. Thus $\depth_R(N)=0$, $\cdd R{\fn}N=\dim_R(N)=1$ and
 $$\ann_R(\h {\depth_R(N)}{\fn}N)=Rx+Ry,$$
 $$\ann_R(\h {\dim_R(N)}{\fn}N)= \ann_R(\Hom R{R/\fp}{E_R(R/\fn)})= \ann_R(R/\fp)=Rx.$$
 On the other hand, $\supp_R(N)=\V(\fp)\cup\V(\fn)=\{\fp, \fn\}$ and since $\fp$ is a minimal element of $\ass_R(N)$,  $C^N_\fp(0)=N_1=Rx/(Rx^2+Rxy)$. Also, it is clear that $C^N_\fn(0)=0$. Now assume that $\Delta $ is a subset of $\supp_R(N)$. Then
 $$\textstyle{\ann_R(N/\bigcap_{\fq\in\Delta}C^N_\fq(0))}=\left\{\begin{array}{lll}
R& \textrm{ if } \Delta=\emptyset,\\
 Rx&\textrm{ if }\Delta=\{\fp\},\\
 Rx^2+Rxy&\textrm{ otherwise. }
 \end{array}\right. $$
 Therefore the following statements hold.

 (i) There is not a subset $\Delta$ of $\supp_R(N)$ such that $\ann_R(\h {\depth_R(N)}{\fn}N)=\ann_R(N/\bigcap_{\fq\in\Delta}C^N_\fq(0)).$

 (ii) By setting $\fa:=\fn$, this example shows that to improve the upper bound for the annihilator of $\h {\grad{R} {\fa}N}{\fa}N$ in Theorem \ref{annh7}, we can not replace $\mass_R(N)$ by $\ass_R(N)$ in the index set $\Sigma$.
\end{ex}
In the following remark, we consider the analogue versions of Theorems \ref{annh2}, \ref{annh4} at $\grad{R} {\fa}N$ instead of at $\cdd R{\fa}N$.

\begin{rem} Let $N$ be a finitely generated $R$-module,  $\fa$ an ideal of $R$ with $N\neq\fa N$, $n:=\grad{R} {\fa}N$ and $c:=\cdd R{\fa}N$.

(i) In Theorem \ref{annh2}, we see that $\ann_R(\h c{\fa}N)\subseteq\ann_R(N/\bigcap_{\fq\in\Sigma}C^N_\fq(0))$, where $\Sigma:=\{\fq\in\supp_R(N): \cdd R{\fa}{R/\fq}=\dim_R(R/\fq)=c \}$.
Now assume that $R$, $\fp$, $\fn$ and  $N$ are as in Example \ref{exa2} and we set $\fa:=\fn$,
\begin{align*}
&\Sigma_1:=\{\fq\in\supp_R(N): \cdd R{\fa}{R/\fq}=\dim_R(R/\fq)=n \},\\
&\Sigma_2:=\{\fq\in\supp_R(N): \cdd R{\fa}{R/\fq}=\depth_R(R/\fq)=n \},\\
&\Sigma_3:=\{\fq\in\supp_R(N): \grad{R} {\fa}{R/\fq}=\dim_R(R/\fq)=n \},\\
&\Sigma_4:=\{\fq\in\supp_R(N): \grad{R} {\fa}{R/\fq}=\depth_R(R/\fq)=n \}.
\end{align*}
By above example, we have $\ann_R(N/\bigcap_{\fq\in\Sigma_i}C^N_\fq(0))=Rx^2+Rxy$ for all $1\leq i\leq 4$ and $\ann_R(\h n{\fa}N)=Rx+Ry$. Hence
$$\textstyle{\ann_R(\h n{\fa}N)\nsubseteq\ann_R(N/\bigcap_{\fp\in\Sigma_i}C^N_\fp(0))}$$
for all $1\leq i\leq 4$.

(ii) Assume again that $R$, $\fp$, $\fn$ and  $N$ are as in Example \ref{exa2} and we set $\fa:=\fn$. Since $\fa$ is the maximal ideal of $R$, $(0:_{\h n{\fa}{N/N'}}\fa)$ has finite length for all submodules $N'$ of $N$. But, by Example \ref{exa2}, there is not a subset $\Sigma'$ of $\supp_R(N)$ such that
$\ann_R(\h n{\fa}N)=\ann_R(N/\bigcap_{\fp\in\Sigma'}C^N_\fp(0))$. Therefore the analogue version of Theorem \ref{annh4} does  hold at $\grad{R} {\fa}N$ instead of at $\cdd R{\fa}N$.
\end{rem}
\begin{prop}\label{prop3}
Let $(R, \fn)$  be a homomorphic image of a Cohen-Macaulay local ring, $N$  a non-zero finitely generated $R$-module and $t\in\N_0$ is such that $\h t{\fn}N\neq 0$. Then
$$\dim_R(R/\ann_R(\h t{\fn}N)\leq t.$$
Equality holds whenever there exists $\fp\in\ass_R(N)$ with $\dim_R(R/\fp)=t$.
\end{prop}
\begin{proof}
Since $\surd(\ann_R(\h t{\fn}N))=\bigcap_{\fp\in\att_R(\h t{\fn}N)}\fp$ and $\att_R(\h t{\fn}N)$ is a finite set, we have $\dim_R(R/\ann_R(\h t{\fn}N)=\dim_R(R/\fp)$ for some $\fp\in\att_R(\h t{\fn}N)$. By \cite[Corollary 1.2]{k}, $R$ is universally catenary and all its formal fibers are Cohen-Macaulay. Thus \cite[Theorem 1.1]{nq} implies that $\fp R_\fp\in\att_{R_\fp}(\h {t-\dim_R(R/\fp)}{\fp R_\fp}{N_\fp})$. Hence $t-\dim_R(R/\fp)\geq 0$ and so $$\dim_R(R/\ann_R(\h t{\fn}N))\leq t.$$
If there exists $\fq\in\ass_R(N)$ with $\dim_R(R/\fq)=t$, then \cite[Corollary 4.9]{s} implies that $\fq\in\att_R(\h t{\fn}N)$ and so $\ann_R(\h t{\fn}N)\subseteq\fq$. Thus $\dim_R(R/\ann_R(\h t{\fn}N))\geq\dim_R(R/\fq)=t$ and so the equality holds.
\end{proof}
The following example shows that there is a local ring $(A, \fn)$ which is a homomorphic image of a complete regular local ring  such that
 $$\dim_A(A/\ann_A(\h {\depth_A(A)}{\fn}{A}))<\depth_A(A)=\depth_A(A/\gam {\fn}A).$$
 Therefore the inequality in Proposition \ref{prop3} may be strict. Also, the analogue version of Lynch's conjecture is not true for $\h {\depth_A(A)}{\fn}A$.
\begin{ex}\label{exa3}
Let $K$ be a field and let $R:=K[[x, y, z, w]]$ be the ring of formal power series over  $K$ in indeterminates $x, y, z, w$. We set $\fm:=(x, y, z, w)$ and   $I:=(x, y)\cap(z, w)$. Then $A:=R/I$ is a local ring with maximal ideal $\fn:=\fm/I$. By \cite[Example 2.8]{f} and the Independence Theorem, we have $\gam {\fn}{A}\cong\gam {\fm}{R/I}=0$ and
$\h 1{\fn}{A}\cong\h 1{\fm}{R/I}\cong R/\fm\cong A/\fn$. Therefore $\depth_A(A)=\depth_A(A/\gam {\fn}A)=1$ and
$$\dim_A(A/\ann_A(\h {\depth_A(A)}{\fn}{A}))=\dim_A(A/\fn)=0.$$
Therefore, in general, $\dim_A(A/\ann_A(\h {\depth_A(A)}{\fn}{A}))$ is not equal $\depth_A(A)$ or $\depth_A(A/\gam {\fn}A)$ and the inequality in Proposition \ref{prop3} may be strict.
\end{ex}
%\begin{ex}
 %Let $(R, \fn)$ be  a complete local ring which is not a Cohen-Macaulay ring and has positive depth (for example we can chose $R:=K[[x, y, z]]/(x, y)\cap(z)$, where $K$ is a field and $x, y, z$ are indeterminates). By Cohen's Structure Theorem $R$ is a homomorphic image of a regular ring  and so Proposition \ref{prop3} implies that $$\dim_R(R/\ann_R(\h {\depth_R(R)}{\fn}R)\leq\depth_R(R)<\dim_R(R)=\dim_R(R/\gam {\fn}R).$$
%Therefore, in general, $\dim_R(R/\ann_R(\h {\depth_R(R)}{\fn}R))\neq\dim_R(R/\gam {\fn}R)$.
%\end{ex}

Let $(R, \fn)$ be a local ring and $N$ a non-zero finitely generated $R$-module. Then, for each $\fp\in\ass_R(N)$,  $\depth_R(N)\leq\dim_R(R/\fp)$. We say that $N$ has {\it maximal depth} if $\depth_R(N)=\dim_R(R/\fp)$ for some $\fp\in\ass_R(N)$. Cohen-Macaulay modules and sequentially Cohen-Macaulay modules have maximal depth; see \cite{r2} for  more details. Now assume in addition that $R$ is a homomorphic image of a Cohen-Macaulay local ring.   Proposition \ref{prop3} shows that $\dim_R(R/\ann_R(\h {\depth_R(N)}\fn N))\leq\depth_R(N)$ and Example \ref{exa3} shows that this inequality may be strict. In  the following corollary we see that the equality holds if $N$ has maximal depth.

\begin{cor} \label{prop4} Let $(R, \fn)$ be a homomorphic image of a Cohen-Macaulay local ring and let $N$  be a non-zero finitely generated $R$-module which has maximal depth. Then
$$\dim_R(R/\ann_R(\h {\depth_R(N)}{\fn} N))=\depth_R(N).$$
\end{cor}
\begin{proof} It is an immediate consequence of Proposition \ref{prop3}.
\end{proof}
%\begin{prop} Let $\fa$ be an ideal of $R$ and  $N$  a finitely generated $R$-module  such that $N\neq \fa N$. Set  $c:=\cdd R {\fa}{N}$ and $J:=\ann_R (\h {c}{\fa}N)$. Then
%$\ann_R (N/J N)=J$; in other word, $N/JN$ is a faithful $R/J$-module.
%\end{prop}
%\begin{proof} We have
%$$J\subseteq\ann_R (N/JN)\subseteq\ann_R (\h {c}{\fa}{N/JN})=\ann_R (\h {c}{\fa}{N}/J\h {c}{\fa}{N})=J$$
%because $J\h {c}{\fa}{N}=0$.
%\end{proof}
\section*{Acknowledgements}
 The author is deeply grateful to the referee for his/her  valuable comments and  suggestions.

\end{document}